\documentclass[a4paper]{amsart}

\usepackage[all]{xy}
\usepackage{amsfonts}
\usepackage{amsmath}
\usepackage{amsthm}
\usepackage{amssymb}
\usepackage{amscd}
\usepackage{graphicx}
\usepackage{caption}
\usepackage{subcaption}
\usepackage[UKenglish]{babel}
\usepackage[T2A]{fontenc}

\usepackage{mathrsfs}       
\usepackage{tikz,tkz-euclide}
\usepackage{tikz-cd}
\usetikzlibrary{arrows}
\usepackage{tikz-3dplot}
\usetikzlibrary{shapes.geometric, calc}

\usepackage{colonequals} 
\usepackage{blkarray, bigstrut} %
\usepackage{comment}

\usepackage{color}\definecolor{darkblue}{rgb}{0,0.1,.5}
\usepackage[colorlinks=true,linkcolor=darkblue, urlcolor=darkblue, citecolor=darkblue]{hyperref}
\usepackage{cleveref}

\def\le{\leqslant}
\def\ge{\geqslant}
\def\leq{\leqslant}
\def\geq{\geqslant}

\def\Z{\mathbb{Z}}

\def \ZZ{{\mathbb{Z}}}

\def\bideg{\mathop{\mathrm{bideg}}}

\DeclareMathOperator{\Hom}{Hom}
\DeclareMathOperator{\Tor}{Tor}

\DeclareMathOperator{\lk}{lk}

\DeclareMathOperator{\rank}{rank}
\DeclareMathOperator{\coker}{coker}
\DeclareMathOperator{\colim}{colim}

\def\HH{\mathit{HH}}
\def\CH{\mathit{CH}}
\def\PH{\mathit{PH}}

\def\sk{\mathcal K}
\def\sK{\mathcal K}

\def\zk{\mathcal Z_{\mathcal K}}
\def\zl{\mathcal Z_{\mathcal L}}

\def\RP{\mathbb{R}{\rm P}}

\def \mc{\mathcal}

\newcommand{\ts}{\textsc}

\newcommand{\hr}[2][]{\hyperref[#2]{#1~\ref{#2}}}

\makeatletter
\newtheorem*{rep@theorem}{\rep@title}
\newcommand{\newreptheorem}[2]{%
\newenvironment{rep#1}[1]{%
 \def\rep@title{#2 \ref{##1}}%
 \begin{rep@theorem}}%
 {\end{rep@theorem}}}
\makeatother

\newreptheorem{theorem}{Theorem}

\newtheorem{theorem}{Theorem}[section]
\newtheorem{proposition}[theorem]{Proposition}
\newtheorem{lemma}[theorem]{Lemma}
\newtheorem{corollary}[theorem]{Corollary}

\theoremstyle{definition}
\newtheorem{example}[theorem]{Example}
\newtheorem{construction}[theorem]{Construction}

\newtheorem{remark}[theorem]{Remark}
\newtheorem{question}[theorem]{Question}

\numberwithin{equation}{section}

\makeatletter
\@namedef{subjclassname@2020}{\textup{2020} Mathematics Subject Classification}
\makeatother

\title{Double cohomology of moment-angle complexes}

\author[Limonchenko]{Ivan Limonchenko}
\address{Mathematical Institute of the Serbian Academy of Sciences
and Arts (SASA), Belgrade, Serbia}
\email{ilimonchenko@gmail.com}

\author[Panov]{Taras Panov}
\address{Department of Mathematics and Mechanics, Moscow
State University, Russia;\newline
Faculty of Computer Science, HSE University, 
Russia;\newline
Institute for Information Transmission Problems, Russian Academy of Sciences, Moscow, Russia}
\email{tpanov@mech.math.msu.su}

\author[Song]{Jongbaek Song}
\address{Department of Mathematics Education, Pusan National University, Pusan, Republic of Korea}
\email{jongbaek.song@gmail.com}

\author[Stanley]{Donald Stanley}
\address{Department of Mathematics and Statistics, University of Regina, Regina, Saskatchewan, Canada}
\email{Donald.Stanley@uregina.ca}

\subjclass[2020]{13F55, 55N10, 55S20, 57S12}

\keywords{moment-angle complex, bicomplex, double cohomology}

\begin{document}

\maketitle

\begin{abstract}
We put a cochain complex structure $\CH^*(\zk)$ on the cohomology of a moment-angle complex $\zk$ and call the resulting cohomology the double cohomology, $\HH^*(\zk)$. We give three equivalent definitions for the differential, and compute $\HH^*(\zk)$ for a family of simplicial complexes containing clique complexes of chordal graphs. 
\end{abstract}

\section{Introduction}
\subsection*{Cohomology of $\zk$.}
Let $\sK$ be a simplicial complex on the set $[m]=\{1,\ldots,m\}$. The associated \emph{moment-angle complex} $\zk$ is a well-studied CW complex with a torus action. 
The cohomology $H^*(\zk)$ can be computed as the cohomology of the Koszul differential graded algebra $\bigl(\Lambda[u_1,\ldots,u_m]\otimes\Z[\sK],d\bigr)$ of the Stanley--Reisner ring (face ring)~$\Z[\sK]$, or as a direct sum $\bigoplus_{I\subset[m]}\widetilde H^*(\sK_I)$ of the reduced simplicial cohomology of full subcomplexes in~$\sK$.

\subsection*{Definition and properties of $\HH^*(\zk)$.}
In this paper we put a cochain complex structure $\CH^*(\zk)$ on $H^*(\zk)$ and 
compute its cohomology $\HH^*(\zk)$, called the \emph{double cohomology} of~$\zk$. We give three equivalent definitions of $\CH^*(\zk)$,
the first corresponding to the induced subcomplex description of $H^*(\zk)$ (Section~3), the second using the Koszul complex (Section~4), and the third coming from the diagonal $S^1$ action on $\zk$ (Section~5).

The double cohomology $\HH^*(\zk)$ is a graded commutative algebra (Theorem \ref{CHDGA}). Furthermore, it satisfies Poincar\'e duality when $\sK$ is a triangulated sphere or a Gorenstein complex (Proposition~\ref{HHPA}). Also, $\HH^*(\zk)$ may have torsion (Example~\ref{rp2ex}).
The double cohomology converts joins into tensor products: \[
\HH^\ast(\mathcal{Z}_{\mathcal{K}\ast \mathcal{L}}) \cong \HH^\ast(\mathcal{Z}_{\mathcal{K}}) \otimes \HH^\ast(\mathcal{Z}_{\mathcal{L}})\] 
(Theorem~\ref{thm_join}, with field coefficients), as in the case of ordinary cohomology of~$\zk$.

\subsection*{Computations of $\HH^*(\zk)$.}
We go on to describe a number of computational results. We start with a breakdown of $\sk$ based on the rank of $\HH^*(\zk)$.

\smallskip

\noindent\textbf{(1)} $\rank \HH^*(\zk)=1.$

\smallskip

It is easy to show that $\zk$ is contractible if and only if $\sk$ is a simplex, which is also equivalent to $\HH^\ast(\zk)=\Z$ (Proposition~\ref{thm_main}). We also show that if $\sk$ is not a simplex then $\HH^*(\zk)$ has zero Euler characteristic and so has even positive rank (Corollary \ref{cor_euler_char_of_HH_for_simplex}).

\smallskip

\noindent\textbf{(2)} $\rank \HH^*(\zk)=2.$

\smallskip

This class of $\sk$ is harder to describe. 
In the case when $\sk$ is a flag simplicial complex, it is shown in~\cite{GPTW} 
that $\zk$ is a wedge of spheres if and only if the one-skeleton $\sk^1$ is a chordal graph. 
We show in Theorem~\ref{thm_main2} that if $\sK$ is the clique complex 
of a chordal graph, then~\textbf{(2)} holds. 
Theorem~\ref{thm_main2} 
also says $\sK$ is the clique complex of a chordal graph exactly when it can be obtained by repeatedly attaching a simplex along a (possibly empty) face, starting from a simplex. 
The simplest examples would be $\sK$ consisting of $m$ discrete points. 
However,~\textbf{(2)} also holds more generally when $\sk$ can be 
written as an attachment of a simplex to any simplicial complex along a proper 
simplex of both. This corresponds to gluing a clique to 
an arbitrary graph along a proper subclique of both, also 
called clique sum, and then taking the clique complex. More precisely 
we have our first main computational result.

\begin{reptheorem}{surgery}
Let $\sk=\sk' \cup_\sigma  \Delta^n$ be a simplicial complex 
obtained from a nonempty simplicial complex $\sk'$ by gluing an $n$-simplex 
along a proper, possibly empty, face $\sigma \in \sk$. Then either $\sk$ is a simplex, or 
\[
  \HH^{-k, 2\ell}(\zk)=\begin{cases} 
    \ZZ & \text{for }(-k,2\ell)=(0,0),~(-1,4);\\
    0 & \text{otherwise}.
  \end{cases}
\]
\end{reptheorem}

We have a few more examples satisfying~\textbf{(2)}. The boundary of an $m$-simplex 
satisfies~\textbf{(2)} (Proposition \ref{prop_boundary_simplex}), but if $m>1$ then 
it has different bidegrees from the ones in 
Theorem~\ref{surgery}. Also we give examples of a non-flag complex 
(Example~\ref{TwoTriangles}) and of a flag complex that is not a union of a complex and a simplex along a simplex (Example~\ref{TwoSquares}) both of which 
satisfy~\textbf{(2)} 
with the same bidegrees as Theorem~\ref{surgery}. These examples can be decomposed as a 
union of two simplicial complexes along a simplex. We wonder if this holds for all such simplicial complexes, 
or if it characterizes such complexes. This would give a homological characterization of graphs that are decomposable as clique sums. 

\smallskip

\noindent\textbf{(3)} $\rank \HH^*(\zk)=4.$

\smallskip

Our second main computational result shows that $\HH^*(\zk)\cong\mathbb Z^4$ for any $m$-cycle with $m\geq 4$ although the grading depends on~$m$. Note that the $4$-cycle $\sk$ is the join $S^0\ast S^0$. So, Theorems~\ref{surgery} and~\ref{thm_join} together show that $\HH^*(\zk)\cong\mathbb Z^4$. For $m\geq 5$, we have the following theorem. 
\begin{reptheorem}{prop_HH_mgon}
Let $\mc{L}$ be an $m$-cycle for $m\geq 5$. Then, the double cohomology of $\mc{Z}_\mc{L}$ is 
\[
\HH^{-k, 2\ell}(\mc{Z}_\mc{L})=
\begin{cases} \Z, & (-k, 2\ell)=(0,0), (-1,4), (-m+3, 2(m-2)), (-m+2, 2m);\\
0, & \text{otherwise}. \end{cases}
\]
\end{reptheorem}

\smallskip

\noindent\textbf{(4)} $\rank \HH^*(\zk)\geq 6.$ 

\smallskip

There are many $\sk$ with larger double cohomology than the one covered by \textbf{(1)}--\textbf{(3)}. For any $n>0$, Theorem~\ref{thm_join} along with Theorem~\ref{surgery} allow us to easily construct $\zk$ such that $\rank \HH^*(\zk)=2^n$. For other ranks, we have Question~\ref{HHrankRealization_quest}.

\subsection*{Motivation}
Our motivation for defining $\HH^*(\zk)$ comes from persistent cohomology $\PH^\ast(-)$. Given a finite pseudometric space $S$, which can be thought of as a data set, we can associate to $S$ a filtered simplicial complex called its Vietoris--Rips complex (other options are \v{C}ech complex, alpha complex, witness complex, etc). It is a family of simplicial complexes $\sk(t)$ depending on a non-negative real number~$t$. 
We can then define the persistent cohomology $\PH^\ast(S)$ of $S$ by the parametrized family $\{H^\ast(\mathcal{K}(t))\}_{t\geq 0}$ together with homomorphisms $H^\ast(\sk(t_2)) \to H^\ast(\sk(t_1))$ for all $t_1$ and $t_2$ with $0\leq t_1< t_2$.

An important property of $\PH^\ast(S)$
is stability, roughly speaking that perturbing $S$ slightly will not change $\PH^\ast(S)$ much. 
In our subsequent work~\cite{BLPSS}, we look at extending this stability to the associated family $\{H^*(\mathcal{Z}_{\sk(t)})\}_{t\geq 0}$ of the cohomology groups of moment-angle complexes. It turns out that $\{H^*(\mathcal{Z}_{\sk(t)})\}_{t\geq 0}$ does not satisfy  stability, but the family of double cohomology $\{\HH^*(\mathcal{Z}_{\sk(t)})\}_{t\geq 0}$ does satisfy stability. It also opens a way to define bigraded barcodes of a data set via bigraded persistent cohomology of the corresponding moment-angle complexes.

\subsection*{Outline}
Section~\ref{sec_prel} contains the background material on the ordinary cohomology of~$\zk$.

In Section~\ref{sec_double_cohom} we introduce the chain complex $\CH_*(\zk)=\bigl(\bigoplus_{I\subset[m]}\widetilde H_*(\sK_I),\partial'\bigr)$ and its cochain version $\CH^*(\zk)$, and define the double homology $\HH_*(\zk)$ and cohomology $\HH^*(\zk)$.

In Section~\ref{bicom} we introduce the bicomplex $\bigl(\Lambda[u_1,\ldots,u_m]\otimes\Z[\sK],d,d'\bigr)$ and show in Theorem~\ref{1db} that its first double cohomology coincides with $\HH^*(\zk)$, i.e.
\[
\HH^\ast(\zk)\cong H(H(\Lambda[u_1,\ldots,u_m]\otimes\Z[\sK],d),d').
\]
We also show that the second double cohomology of the bicomplex is zero (Proposition~\ref{2db}). This implies that the first spectral sequence of the bicomplex has $E_2=\HH^*(\zk)$ and converges to zero. The differentials in this spectral sequence are related to the higher cohomology operations for moment-angle complexes studied recently in~\cite{am-br}. The spectral sequence also implies that the Euler characteristic of $\HH^*(\zk)$ is zero, unless $\sK$ is a full simplex (Corollary~\ref{cor_euler_char_of_HH_for_simplex}).

In Section~\ref{relta} we relate the double cohomology to the torus action on~$\zk$ by showing that the differential $d'$ can be identified with the derivation of $H^*(\zk)$ arising from the diagonal circle action.

Sections~\ref{sec_techniques} and \ref{sec_mcycle} contain the main computational results on the double cohomology of $\zk$. We give more illustrative examples of computations in Section~\ref{sec_examples}, together with some open questions.

\subsection*{Acknowledgements}
The authors thank the Fields Institute for Research in Mathematical Sciences for the opportunity to work on this research project during the Thematic Program on Toric Topology and Polyhedral Products.
We are grateful to Tony Bahri for making us work together and inspirational discussions at the early stages of this project. 

Limonchenko was supported by the Serbian Ministry of Science, Innovations and Technological Development through the Mathematical Institute of the Serbian Academy of Sciences and Arts.
Panov has been funded within the framework of the HSE University Basic Research Program.
Song was supported by Basic Science Research Program through the National Research Foundation of Korea (NRF) funded by the Ministry of Education (NRF-2018R1D1A1B07048480) and a KIAS Individual Grant (MG076101) at Korea Institute for Advanced Study. He is also supported by Pusan National University Research Grant, 2023.
Stanley has funding support from NSERC (RPGIN-05466-2020). We thank the anonymous referee for the valuable comments and useful suggestions on improving the text.

\section{Preliminaries}\label{sec_prel}
Let $\sK$ be a simplicial complex on the set $[m]=\{1,2,\ldots,m\}$. We refer to a subset $I=(i_1,\ldots,i_k)\subset[m]$ that is contained in $\sK$ as a \emph{simplex}. A one-element simplex $\{i\}\in\sK$ is a \emph{vertex}. We also assume that $\varnothing\in\sK$ and, unless explicitly stated otherwise, that $\sK$ contains all one-element subsets $\{i\}\in[m]$ (that is, $\sK$ is a simplicial complex \emph{on the vertex set~$[m]$} without \emph{ghost vertices}).

Denote by $\ts{cat}(\sk)$ the face category of $\sk$, with objects $I\in\sK$ and morphisms $I\subset J$. For each subset $I\in\sk$, we consider the following topological space  
\[
  (D^2, S^1)^I \colonequals \{(z_1, \dots, z_m)\in (D^2)^m \colon |z_j|=1 \text{ if } j\notin I\}\subset (D^2)^m.
\]
Note that $(D^2, S^1)^I$ is a natural subspace of  $(D^2, S^1)^J$ whenever $I\subset J$. Hence, we have a diagram 
\[
\mathscr{D}_\sk \colon \ts{cat}(\sk) \to \ts{top}
\]
mapping $I \in \sk$ to $(D^2, S^1)^I$. The \emph{moment-angle complex} corresponding to $\sk$ is 
\[
  \zk\colonequals \colim \mathscr{D}_K=\bigcup_{I\in\sK}(D^2,S^1)^I\subseteq(D^2)^m. 
\]
We refer to \cite[Chapter~4]{bu-pa15} for more details and examples. 

To each simplicial complex $\sk$, one can define the \emph{face ring} by 
\[
\mathbb{Z}[\sk]\colonequals \mathbb{Z}[v_1, \dots, v_m]/\mathcal{I}_\sk, \]
where $\mathcal{I}_\sk$ is the ideal generated by monomials $\prod_{i\in I} v_i$ for which $I\subset[m]$ is not a simplex of $\sK$. 

The following theorem summarizes several presentations of the cohomology ring $H^\ast(\zk)$. We consider homology and cohomology with coefficients in $\Z$, although the results are generalised easily to any principal ideal domain.

\begin{theorem}[\cite{b-b-p04}, \cite{bu-pa00}]\label{zkcoh}
There
are isomorphisms of bigraded commutative algebras
\begin{align}
  H^*(\zk)&\cong\Tor_{\Z[v_1,\ldots,v_m]}\bigl(\Z[\sk],\Z\bigr) \nonumber \\
  &\cong H\bigl(\Lambda[u_1,\ldots,u_m]\otimes\Z[\sk],d\bigr) \label{eq_Koszul_cohom}\\
  &\cong \bigoplus_{I\subset[m]}\widetilde H^*(\sk_I). \label{eq_Hochster_decomp}
\end{align}
Here, \eqref{eq_Koszul_cohom} is the cohomology of the
bigraded algebra with $\bideg u_i=(-1,2)$, $\bideg v_i=(0,2)$ and differential of bidegree $(1,0)$ given by
$du_i=v_i$, $dv_i=0$ (the Koszul complex). In \eqref{eq_Hochster_decomp},
$\widetilde H^*(\sk_I)$ denotes the reduced simplicial cohomology
of the full subcomplex $\sk_I\subset \sk$ (the restriction of $\sk$
to $I\subset[m]$). The last isomorphism is the sum of isomorphisms
\[
  H^p(\zk)\cong
  \sum_{I\subset[m]}\widetilde H^{p-|I|-1}(\sk_I),
\]
and the ring structure is given by the maps
\[
  H^{p-|I|-1}(\sk_I)\otimes H^{q-|J|-1}(\sk_J)\to
  H^{p+q-|I|-|J|-1}(\sk_{I\cup J})
\]
which are induced by the canonical simplicial maps $\sk_{I\cup
J}\to\sk_I\mathbin{*}\sk_J$ for $I\cap J=\varnothing$ and zero otherwise.
\end{theorem}

Isomorphism~\eqref{eq_Hochster_decomp} is often referred to as the \emph{Hochster decomposition}, as it comes from Hochster's theorem describing $\Tor_{\Z[v_1,\ldots,v_m]}(\Z[\sk],\Z)$ as a sum of the cohomologies of full subcomplexes.

The bigraded components of the cohomology of $\zk$ are given by
\[
  H^{-k,2\ell}(\zk)\cong\bigoplus_{I\subset[m]\colon |I|=\ell}\widetilde H^{\ell-k-1}(\sK_I),\quad
  H^p(\zk)=\bigoplus_{-k+2\ell=p}H^{-k,2\ell}(\zk).
\]

Consider the following quotient of the Koszul ring $\Lambda[u_1,\ldots,u_m]\otimes\Z[\sk]$:
\begin{equation}\label{rk}
  R^*(\sK)=\Lambda[u_1,\ldots,u_m]\otimes\Z[\sK]\bigr/(v_i^2=u_iv_i=0,\;
  1\le i\le m).
\end{equation}
Then $R^*(\sK)$ has finite rank as an abelian group, with a
basis of monomials $u_Jv_I$ where $J\subset[m]$, $I\in\sK$ and
$J\cap I=\varnothing$. Furthermore, $R^*(\sK)$ can be identified with the cellular cochains $C^*(\zk)$ of $\zk$ with appropriate cell decomposition, the quotient ideal $(v_i^2=u_iv_i=0,\;1\le i\le m)$ is $d$-invariant and acyclic, and there is a ring isomorphism
\[
  H^*(\zk)\cong H\bigl(R^*(\sk),d\bigr),
\]
see~\cite[\S4.5]{bu-pa15} for the details.

The algebras assigned to a simplicial complex $\sK$ above have the following functorial properties, which follow easily from the construction, see~\cite[Proposition~4.5.5]{bu-pa15}.

\begin{proposition}\label{funcH}
Let $\sk$ be a simplicial complex on $m$ vertices, and 
let $\mathcal L\subset\sK$ be a subcomplex on $\ell$ vertices. 
The inclusion $\mathcal L\subset\sK$ induces an inclusion $\mathcal Z_{\mathcal L}\to\zk$ and homomorphisms of (differential) graded algebras
\begin{itemize}
\item[(a)] $\Z[\sk]\to\Z[\mathcal L]$,
\item[(b)] $\bigl(\Lambda[u_1,\ldots,u_m]\otimes\Z[\sk],d\bigr)\to  \bigl(\Lambda[u_1,\ldots,u_\ell]\otimes\Z[\mathcal L],d\bigr)$,
\item[(c)] $\bigl(R^*(\sK),d\bigr)\to\bigl(R^*(\mathcal L),d\bigr)$,
\item[(d)] $H^*(\mathcal Z_{\mathcal K})\to H^*(\mathcal Z_{\mathcal L})$,
\end{itemize}
defined by sending $u_i,v_i$ to $0$ for $i\notin[\ell]$.

Furthermore, if $\sK_I$ is a full subcomplex for some $I\subset[m]$, then we have a retraction $\zk\to\mathcal Z_{\mathcal K_I}$ and homomorphisms
\begin{itemize}
\item[(e)] $\Z[\sk_I]\to\Z[\sK]$,
\item[(f)] $H^*(\mathcal Z_{\mathcal K_I})\to H^*(\zk)$.
\end{itemize}
There are also homology versions of these homomorphisms, which map between $H_*$ in the opposite direction.
\end{proposition}

In~\cite[Proposition~4.5.5]{bu-pa15} a more general functorial property was established, with respect to arbitrary simplicial maps $\mathcal L\to\sK$ (not just inclusions). We shall use this extended functoriality in Section~\ref{relta}.

\section{Double (co)homology}\label{sec_double_cohom}
In this section, we define the double homology and the double cohomology of a moment-angle complex $\zk$. We continue working with $\Z$ coefficients.
\subsection{Double homology}\label{homology_HH}
We have
\[
  H_p(\zk)\cong
  \bigoplus_{I\subset[m]} 
  \widetilde H_{p-|I|-1}(\sk_I),
\]
similarly to the cohomological Hochster decomposition of Theorem~\ref{zkcoh}. Given $j\in[m]\setminus I$, consider the homomorphism
\[
  \phi_{p;I,j}\colon\widetilde H_p(\sk_I)\to \widetilde H_p(\sk_{I\cup\{j\}})
\]
induced by the inclusion $\sk_I\hookrightarrow\sk_{I\cup\{j\}}$.
Then we define, for a fixed $I\subset [m]$,
\[
\partial'_p=(-1)^{p+1}\sum_{j\in[m]\setminus I}\varepsilon(j,I)\,\phi_{p;I,j} \colon \widetilde H_p(\sk_I)\to 
\bigoplus_{j\in[m]\setminus I} 
\widetilde H_p(\sk_{I\cup\{j\}}),
\]
where 
\[
  \varepsilon(j,I)=(-1)^{\#\{i\in I\colon i<j\}}.
\]
The extra sign $(-1)^{p+1}$ is chosen so that $\partial'$ together with the simplicial boundary $\partial$ satisfy the bicomplex relation $\partial\partial'=-\partial'\partial$, see Section~\ref{bicom}.

\begin{lemma}\label{lem_HH_boundary}
The homomorphism $\partial'_p\colon
\bigoplus_{I\subset[m]} 
\widetilde H_p(\sk_I)\to
\bigoplus_{I\subset[m]} 
\widetilde H_p(\sk_I)$ satisfies $(\partial'_p)^2=0$. 
\end{lemma}
\begin{proof}
We have
\begin{align*}
\partial'_p\partial'_p&=\sum_{k\in[m]\setminus(I\cup j)}\varepsilon(k,I\cup\{j\})\,
  \phi_{p;I\cup\{j\},k}\Bigl(\sum_{j\in[m]\setminus I}\varepsilon(j,I)\,\phi_{p;I,j}\Bigr)
  \\
  &=\sum_{j,k\in[m]\setminus I,\,j\ne k}\bigl(\varepsilon(k,I\cup\{j\})\varepsilon(j,I)+
  \varepsilon(j,I\cup\{k\})\varepsilon(k,I)\bigr)\,\phi_{p;I,j,k},
\end{align*}
where $\phi_{p;I,j,k}\colon\widetilde H_p(\sk_I)\to \widetilde H_p(\sk_{I\cup\{j,k\}})$ is the homomorphism 
induced by the inclusion $\sk_I\hookrightarrow\sk_{I\cup\{j,k\}}$. Now the required identity follows from
\begin{multline*}
  \#\{r\in I\cup\{j\}\colon r<k\}+\#\{s\in I\colon s<j\}\\
  =1+\#\{s\in I\cup\{k\}\colon s<j\}+\#\{r\in I\colon r<k\}\mod 2.\qedhere
\end{multline*}
\end{proof}

We therefore have a chain complex 
\[
  \CH_\ast(\zk)\colonequals (H_\ast(\zk), \partial'),
\] 
where 
\[
\partial' \colon \widetilde{H}_{-k, 2\ell}(\zk) \to \widetilde{H}_{-k-1, 2\ell+2}(\zk)
\]
with respect to the following bigraded decomposition of $H_p(\zk)$ 
\[
  H_p(\zk)=\bigoplus_{-k+2\ell=p} H_{-k, 2\ell}(\zk),\quad
  H_{-k, 2\ell}(\zk)\cong\bigoplus_{I\subset [m]\colon|I|=\ell}
  \widetilde{H}_{\ell-k-1}(\mathcal{K}_I).
\]
We define the bigraded \emph{double homology} of $\zk$ by
\[
  \HH_\ast(\zk)=H(H_\ast(\zk), \partial').
\]

\begin{remark}
Although $\partial'$ increases the total degree of $H_*(\zk)$ by~$1$, we refer to $(\CH_*(\zk),\partial')$ as a chain (rather than a cochain) complex. The reason is that $\partial$ and $\partial'$ satisfy the bicomplex relation $\partial\partial'=-\partial'\partial$, see Section~\ref{bicom}. The same remark refers to the ``cochain'' complex $(\CH^*(\zk),d')$ considered next.
\end{remark}

\subsection{Double cohomology}\label{subsec_HH_Cohomology}
For the cohomological version, given $i\in I$, consider the homomorphism 
\begin{equation}\label{eq_cohom_differential_component}
  \psi_{p;i,I}\colon\widetilde H^p(\sk_I)\to \widetilde H^p(\sk_{I\setminus\{i\}})
\end{equation}
induced by the inclusion $\sk_{I\setminus\{i\}}\hookrightarrow\sk_I$, and 
\begin{equation}\label{d'simp}
  d'_p=(-1)^{p+1}\sum_{i\in I}\varepsilon(i,I)\psi_{p;i,I}.
\end{equation}

We define a map $d'\colon H^*(\zk)\to H^*(\zk)$ using the decomposition given in \eqref{eq_Hochster_decomp}  together with \eqref{eq_cohom_differential_component} and \eqref{d'simp}; it acts on the bigraded cohomology of 
$\zk$ as follows:
\begin{equation}\label{eq_d'_HZK}
  d'\colon H^{-k,2\ell}(\zk)\to H^{-k+1,2\ell-2}(\zk).
\end{equation}
A computation similar to that of Lemma \ref{lem_HH_boundary} shows that $(d')^2=0$, which turns $H^\ast(\zk)$ into a cochain complex
\begin{equation}\label{chcoho}
  \CH^\ast(\zk)\colonequals (H^\ast(\zk), d').
\end{equation}
We define the bigraded \emph{double cohomology} of $\zk$ by
\begin{equation*}\label{eq_HH_via_cohomology}
  \HH^*(\zk)=H\bigl(H^*(\zk),d'\bigr).
\end{equation*}

We end this section with an example where $\HH_{-k, 2\ell}(\zk)$ is not isomorphic to $\HH^{-k, 2\ell}(\zk)$, as in the case of ordinary cohomology in the presence of torsion. It also shows that the usual universal coefficient theorem does not apply in general. 

\begin{example}\label{rp2ex}
Let $\mathcal{K}$ be the minimal $6$-vertex triangulation of $\RP^2$, shown in Figure~\ref{fig_min_triangulation_RP2}.
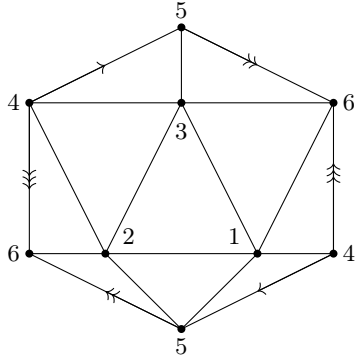
\begin{figure}
\begin{tikzpicture}[scale=0.5]
\draw (4,0)--(8,2)--(8,6)--(4,8)--(0,6)--(0,2)--cycle;
\fill [black] (4,0) circle (3pt) node[below] at (4,0){\small5};
\fill [black] (8,2) circle (3pt) node[right] at (8,2){\small4};
\fill [black] (8,6) circle (3pt) node[right] at (8,6){\small6};
\fill [black] (4,8) circle (3pt) node[above] at (4,8){\small5};
\fill [black] (0,6) circle (3pt) node[left] at (0,6){\small4};
\fill [black] (0,2) circle (3pt) node[left] at (0,2){\small6};

\fill [black] (2,2) circle (3pt) node[above right] at (2.2,2){\small2};
\fill [black] (6,2) circle (3pt) node[above left] at (5.8,2){\small1};
\fill [black] (4,6) circle (3pt) node[below] at (4,5.7){\small3};

\draw (0,2)--(8,2);
\draw (0,6)--(8,6);
\draw (0,6)--(2,2)--(4,6)--(6,2)--(8,6);
\draw (4,8)--(4,6);
\draw (2,2)--(4,0)--(6,2);

\draw[->] (0,6)--(2,7);
\draw[->] (8,2)--(6,1);

\draw[->>] (4,0)--(2,1);
\draw[->>] (4,8)--(6,7);

\draw[->>>] (8,2)--(8,4.4);
\draw[->>>] (0,6)--(0,3.7);

\end{tikzpicture}
\caption{A minimal triangulation of $\mathbb{R}P^2$.}
\label{fig_min_triangulation_RP2}
\end{figure}
In this case,
\[
    H^{-k, 2\ell}(\zk)\cong\begin{cases}
    \widetilde H^2(\mathcal{K})\cong \ZZ_2,  & (-k,2\ell)=(-3,12);\\
    \bigoplus_{1\le i \le 6} 
    \widetilde H^1(\mathcal{K}_{[6]\setminus \{i\}}) \cong \ZZ^6, & (-k, 2\ell)=(-3, 10);\\
    \bigoplus_{1\le i<j \le 6} 
    \widetilde H^1(\mathcal{K}_{[6]\setminus \{i,j\}}) \cong \ZZ^{15}, & (-k, 2\ell)=(-2, 8);\\
     \bigoplus_{\substack{1\le i<j<k \le 6 \\ \{i,j,k\} \notin \mathcal{K}  }} 
     \widetilde H^1(\mathcal{K}_{\{i,j,k\}}) \cong  \ZZ^{10}, & (-k, 2\ell)=(-1, 6);\\
     \widetilde H^{-1}(\mathcal{K}_\varnothing) \cong \ZZ &(-k, 2\ell)=(0,0);\\
     0, & \text{otherwise}. 
    \end{cases}
\]
Hence, to calculate $\HH^\ast(\zk)$ we only need to consider 
\begin{align}
   0& \xrightarrow{d'} H^{-3,12}(\zk) \xrightarrow{d'} 0, \label{eq_cohom_H312}\\
     0& \xrightarrow{d'} H^{-3,10}(\zk) \xrightarrow{d'} H^{-2,8}(\zk) \xrightarrow{d'} H^{-1, 6}(\zk) \xrightarrow{d'} 0.\label{eq_cohom_H310}
\end{align}
The first sequence \eqref{eq_cohom_H312} implies that the torsion part  $\widetilde{H}^2(\mathcal{K})=\mathbb{Z}_2$ survives in the double cohomology, which gives 
\[
  \HH^9(\zk)=\HH^{-3,12}(\zk)=\mathbb{Z}_2.
\]  

For the homology computation, we observe that $H_{-k, 2\ell}(\zk)$ is isomorphic to $H^{-k, 2\ell}(\zk)$ for all $(-k, 2\ell)$ except $(-k, 2\ell)=(-3, 12)$ and $(-4,12)$, in which case
\[
  H_{-3,12}(\zk)\cong\widetilde H_2(\mathcal K)=0, \quad
  H_{-4,12}(\zk)\cong\widetilde H_1(\mathcal K)\cong\Z_2.
\] 
Hence, to calculate $\HH_\ast(\zk)$, we consider the chain complex
\[
\begin{tikzcd}[column sep=small]
0 \arrow{r}{\partial'} & 
H_{-1, 6}(\zk) \arrow{r}{\partial'} 
&
H_{-2, 8}(\zk) \arrow{r}{\partial'} 
&
H_{-3, 10}(\zk) \arrow{r}{\partial'}
&
H_{-4, 12}(\zk) \arrow{r}{\partial'} 
& 0.
\end{tikzcd}
\]
We observe that $\partial'\colon H_{-3, 10}(\zk) \to H_{-4, 12}(\zk)$ is surjective, because 
\begin{equation}\label{eq_H_310}
H_{-3, 10}(\zk)\cong\bigoplus_{1\leq i \leq 6}\widetilde{H}_1(\mathcal{K}_{[6]\setminus \{i\}}) \cong \bigoplus_{1\leq i \leq 6}\widetilde{H}_1(\RP^1)
\end{equation}
and each component of \eqref{eq_H_310} surjects naturally onto $H_{-4,12}(\zk)=\widetilde{H}_1(\RP^2)$. Hence, the torsion part $\widetilde{H}_1(\mathcal{K})=\mathbb{Z}_2$ vanishes in the double homology, and we obtain 
\[
  \HH_{-3,12}(\zk)=\HH_{-4,12}(\zk)=0.
\]  
\end{example}

\section{The bicomplexes}
\label{bicom}
Here we provide an algebraic description of $\HH^\ast(\zk)$ by showing it to be isomorphic to the double cohomology of a bicomplex arising from the Koszul algebra~\eqref{eq_Koszul_cohom}.

\begin{construction}
Given $I\subset[m]$, let $C^p(\sK_I)$ denote the $p$th simplicial cochain group of $\sK_I$. 
Denote by $\alpha_{L,I}\in C^{q-1}(\sK_I)$ the
basis cochain corresponding to an oriented simplex
$L=(\ell_1,\ldots,\ell_q)\in\sK_I$; it takes value~$1$ on $L$ and
vanishes on all other simplices. The simplicial coboundary map (differential) $d$ is defined by
\[
  d\alpha_{L,I}=\sum_{j\in I\setminus L,\,L\cup\{j\}\in\sK}\varepsilon(j,L)\alpha_{L\cup\{j\},I}.
\]
The differential $d'$ can be defined on the cochains by the same formula~\eqref{d'simp}, with
$\psi_{p;i,I}\colon C^p(\sk_I)\to C^p(\sk_{I\setminus\{i\}})$
induced by the inclusion $\sk_{I\setminus\{i\}}\hookrightarrow\sk_I$.

We introduce the second differential $d'$ of bidegree $(1,-2)$ on the Koszul bigraded ring $\Lambda[u_1,\ldots,u_m]\otimes\Z[\sk]$ by setting
\begin{equation}\label{d'uv}
  d'u_j=1,\quad d'v_j=0,\quad\text{ for }j=1,\ldots,m,
\end{equation}
and extending by the Leibniz rule. Explicitly, the differential $d'$ is defined on square-free monomials $u_Jv_I$ by
\[
  d'(u_Jv_I)=\sum_{j\in J}\varepsilon(j,J)u_{J\setminus\{j\}}v_I,
  \quad d'(v_I)=0.
\]

The differential $d'$ is also defined by the same formula on the submodule $R^*(\sK)\subset \Lambda[u_1,\ldots,u_m]\otimes\Z[\sk]$ generated by the monomials $u_Jv_I$ with $J\cap I=\varnothing$. However, the ideal $\bigr(v_i^2=u_iv_i=0,\; 1\le i\le m)$ is not $d'$-invariant, so 
$(R(\sK),d')$ is not a differential graded algebra.
\end{construction}

\begin{lemma}
With $d$ and $d'$ defined above, $\bigl(\bigoplus_{I\subset[m]} C^*(\sK_I),d,d'\bigr)$,
$\bigl(\Lambda[u_1,\ldots,u_m]\otimes\Z[\sk],d,d'\bigr)$ and $\bigl(R^*(\sK),d,d'\bigr)$ are bicomplexes, that is, $d$ and $d'$ satisfy $dd'=-d'd$.
\end{lemma}
\begin{proof}
For a simplicial cochain $\alpha_{L,I}\in C^{q-1}(\sK_I)$,
\begin{align*}
  dd'(\alpha_{L,I})&=d\bigl((-1)^q\sum_{i\in I\setminus L}\varepsilon(i,I)\alpha_{L,I\setminus\{i\}}\bigr)=
  (-1)^q\!\!\!\!\!\!\!\!\!\!\!\!\sum_{i\in I\setminus L,\,j\in I\setminus(L\cup\{i\})}\!\!\!\!\!\!\!\!\!\!\!\!
  \varepsilon(i,I)\varepsilon(j,L)\alpha_{L\cup\{j\},I\setminus\{i\}},\\
  d'd(\alpha_{L,I})&=d'\bigl(\sum_{j\in I\setminus L}\varepsilon(j,L)\alpha_{L\cup \{j\},I}\bigr)=
  (-1)^{q+1}\!\!\!\!\!\!\!\!\!\!\sum_{j\in I\setminus L,\,i\in I\setminus(L\cup\{j\})}
  \!\!\!\!\!\!\!\!\!\!
  \varepsilon(j,L)\varepsilon(i,I)\alpha_{L\cup\{j\},I\setminus\{i\}},
\end{align*}
so that $dd'=-d'd$.

For a monomial $u_Jv_I$ in the Koszul ring $\Lambda[u_1,\ldots,u_m]\otimes\Z[\sk]$ (or in $R^*(\sk)$),
\begin{align*}
  dd'(u_Jv_I)&=d\bigl(\sum_{j\in J}\varepsilon(j,J)u_{J\setminus\{j\}}v_I\bigr)=
  \sum_{j,k\in J,\,j\ne k}\varepsilon(j,J)\varepsilon(k,J\setminus\!\{j\})
  u_{J\setminus\{j,k\}}v_{I\cup\{k\}},\\
  d'd(u_Jv_I)&=d'\bigl(\sum_{k\in J}\varepsilon(k,J)u_{J\setminus\{k\}}v_{I\cup\{k\}}\bigr)=
  \!\!\!\sum_{j,k\in J,\,j\ne k}\!\!\!\varepsilon(k,J)\varepsilon(j,J\setminus\!\{k\})
  u_{J\setminus\{j,k\}}v_{I\cup\{k\}}.
\end{align*}
Now the identity $dd'=-d'd$ follows from
\begin{multline*}
  \#\{p\in J\colon p<j\}+\#\{q\in J\setminus\!\{j\}\colon q<k\}\\
  =1+\#\{q\in J\colon q<k\}+\#\{p\in J\setminus\!\{k\}\colon p<j\}\mod 2.  \qedhere
\end{multline*}
\end{proof}

By construction, $\HH^*(\zk)$ is the first double cohomology of the bicomplex 
$\bigl(\bigoplus_{I\subset[m]} C^*(\sK_I),d,d'\bigr)$:
\[
  \HH^*(\zk)=H\bigl(H\bigl(\bigoplus_{I\subset[m]} C^*(\sK_I),d\bigr),d'\bigr).
\]

\begin{theorem}\label{1db}
The bicomplexes $\bigl(\bigoplus_{I\subset[m]} C^*(\sK_I),d,d'\bigr)$ and $\bigl(R^*(\sK),d,d'\bigr)$ are isomorphic. Therefore,  $\HH^*(\zk)$ is isomorphic to the first double cohomology of the bicomplex
$\bigl(\Lambda[u_1,\ldots,u_m]\otimes\Z[\sk],d,d'\bigr)$:
\[
  \HH^*(\zk)\cong H\bigl(H\bigl(\Lambda[u_1,\ldots,u_m]\otimes\Z[\sk],d\bigr),d'\bigr).
\]
\end{theorem}
\begin{proof}
Define a homomorphism
\[
\begin{aligned}
  f\colon C^{q-1}(\sK_I)&\longrightarrow R^{q-|I|,2|I|}(\sK),\\
  \alpha_{L,I}&\longmapsto \varepsilon(L,I)\,u_{I\setminus L}v_L,
\end{aligned}
\]
where $\varepsilon(L,I)$ is the sign given by
\[
  \varepsilon(L,I)=\prod_{i\in L}\varepsilon(i,I)=(-1)^{\sum_{\ell\in L}\#\{i\in I\colon i<\ell\}}.
\]
Clearly, $f$ is an isomorphism of free abelian groups. 
We claim that $f$ defines an isomorphism of bicomplexes
\[
  f\colon\bigl(\bigoplus_{I\subset[m]}C^*(\sK_I),d,d'\bigr)\longrightarrow \bigl(R^*(\sK),d,d'\bigr).
\] 
The fact that $f$ commutes with $d$ is verified in~\cite[Theorem~3.2.9]{bu-pa15}. To see that $f$ commutes with $d'$ consider the square
\[
\xymatrix{
  C^{q-1}(\sK_I) \ar[r]^f \ar[d]_{d'} 
  & R^{q-|I|,\,2|I|}(\sK) \ar[d]^{d'} \\
  \bigoplus_{i\in I}C^{q-1}(\sK_{I\setminus\{i\}}) \ar[r]^f & R^{q-|I|+1,\,2|I|-2}(\sK)
} 
\]
An element $\alpha_{L,I}\in C^{q-1}(\sK_I)$ with $|L|=q$ is mapped as follows:
\[
\xymatrix{
  \alpha_{L,I} \ar@{|->}[r]^-f \ar@{|->}[dd]_{d'} & 
  \varepsilon(L,I)u_{I\setminus L}v_L \ar@{|->}[d]^{d'}\\
  &\sum_{i\in I\setminus L}\varepsilon(L,I)\varepsilon(i,I\setminus\! L)u_{I\setminus(L\cup\{i\})}v_L\\
  (-1)^{|L|}\sum_{i\in I\setminus L}\varepsilon(i,I)\alpha_{L,I\setminus\!\{i\}} \ar@{|->}[r]^-f 
  & (-1)^{|L|}\sum_{i\in I\setminus L}\varepsilon(i,I)
  \varepsilon(L,I\setminus\!\{i\})u_{I\setminus(L\cup\{i\})}v_L
}  
\]
We therefore need to check that
\[
   \varepsilon(L,I)\varepsilon(i,I\setminus\! L)=(-1)^{|L|}\varepsilon(i,I)\varepsilon(L,I\setminus\!\{i\}),
\]
which is equivalent to
\begin{multline*}
  \sum_{\ell\in L}\#\{p\in I\colon p<\ell \}+\#\{q\in I\setminus\!L\colon q<i\}
  \\=|L|+\#\{q\in I\colon q<i\}+\sum_{\ell\in L}\#\{p\in I\setminus\!\{i\}\colon p<\ell\}\mod 2.
\end{multline*}
This identity rewrites as
\[
  \#\{\ell\in L\colon i<\ell\}=|L|+\#\{q\in L\colon q<i\}\mod 2, 
\]
which clearly holds.
\end{proof}

\begin{theorem}\label{CHDGA}
$\CH^\ast(\zk)= (H^\ast(\zk), d')$ is a commutative differential graded algebra. 
Therefore, the double cohomology $\HH^*(\zk)$ is a graded commutative algebra, with the product induced from the cohomology product on~$H^*(\zk)$.
\end{theorem}
\begin{proof}
We need to verify that the second differential $d'$ satisfies the Leibniz formula with respect to the product in $H^*(\zk)$. Take elements
\[
  \alpha,\beta\in H^*(\zk)= H\bigl(\Lambda[u_1,\ldots,u_m]\otimes\Z[\sk],d\bigr).
\]
and choose their representing $d$-cocycles
\[
  a,b\in \Lambda[u_1,\ldots,u_m]\otimes\Z[\sk],
\]
so that $\alpha=[a]$, $\beta=[b]$.
The bicomplex relations $dd'=-d'd$ imply that $[d'a]=d'[a]$ in~$H^*(\zk)$. By the definition of $d'$, it satisfies the Leibniz formula with respect to the product in $\Lambda[u_1,\ldots,u_m]\otimes\Z[\sk]$:
\[
  d'(ab)=(d'a)b+(-1)^{|a|}a(d'b).
\]
Hence, we have 
\begin{multline*}
  d'(\alpha\beta)=d'([a][b])=d'([ab])=[d'(ab)]\\=
  [(d'a)b+(-1)^{|a|}a(d'b)]=
  (d'\alpha)\beta+(-1)^{|\alpha|}\alpha(d'\beta),
\end{multline*}
as needed.
\end{proof}

\begin{proposition}\label{2db}\ 
\begin{itemize}
\item[(a)]
For any $\sK$, the $d'$-cohomology of $\Lambda[u_1,\ldots,u_m]\otimes\Z[\sk]$ is zero:
\[
  H\bigl(\Lambda[u_1,\ldots,u_m]\otimes\Z[\sk],d'\bigr)=0.
\]

\item[(b)]
If $\sK\ne\varDelta^{m-1}$ (the full simplex on~$[m]$), then the $d'$-cohomology of the bicomplexes 
$\bigoplus_{I\subset[m]} C^*(\sK_I)$ and $R^*(\sK)$ is zero:
\[
  H\bigl(\bigoplus_{I\subset[m]} C^*(\sK_I),d'\bigr)=H\bigl(R^*(\sK),d'\bigr)=0.
\]
Therefore, the second double cohomology and the total cohomology of the bicomplexes 
$\bigl(\bigoplus_{I\subset[m]} C^*(\sK_I),d,d'\bigr)$ and $\bigl(R^*(\sK),d,d'\bigr)$ are zero unless 
$\sK=\varDelta^{m-1}$.

\item[(c)] If $\sK=\varDelta^{m-1}$, then the only nonzero $d'$-cohomology group of 
$\bigoplus_{I\subset[m]} C^*(\sK_I)$ and $R^*(\sK)$ is $H^{2m}\cong\Z$, represented by $\alpha_{[m],[m]}$ and $v_1\cdots v_m$, respectively.
\end{itemize}
\end{proposition}
\begin{proof}
For (a), the differential $d'$ does not change the $v$-part of the monomials $u_Jv_{i_1}^{p_1}\cdots v_{i_k}^{p_k}\in \Lambda[u_1,\ldots,u_m]\otimes\Z[\sk]$. Therefore, the complex $(\Lambda[u_1,\ldots,u_m]\otimes\Z[\sk],d'\bigr)$ splits into the sum over monomials in $\Z[\sk]$ of the exterior complexes $\bigl(\Lambda[u_1,\ldots,u_m],d'\bigr)$ with $d'u_I=\sum_{i\in I}\varepsilon(i,I)u_{I\setminus\{i\}}$, which are clearly acyclic.

\smallskip

For (b) and (c), the argument is similar: the monomials $u_Jv_I\in R^*(\sK)$ with $J\cap I=\varnothing$ and fixed $I\in\sK$ span a $d'$-subcomplex isomorphic to the exterior complex 
$\bigl(\Lambda[u_j\colon j\in[m]\setminus\!I],d'\bigr)$. It is acyclic unless $I=[m]$, which leaves us with the monomial $v_{[m]}=v_1\cdots v_m$. It is nonzero in $R^*(\sK)$ if only if 
$\sK=\varDelta^{m-1}$. 

The statement about the total cohomology follows by considering the standard spectral sequence of the bicomplex.
\end{proof}

\begin{corollary}\label{cor_euler_char_of_HH_for_simplex}
If $\sK\ne\varDelta^{m-1}$, then the Euler characteristic of $\HH^*(\zk)$ is zero.
\end{corollary}
This also follows from the fact that the Euler characteristic of $H^*(\zk)$ is zero unless 
$\sK=\varDelta^{m-1}$, see~\cite[Corollary~4.6.3]{bu-pa15}.

\begin{remark}
Proposition~\ref{2db} (a) does not imply that the total cohomology of the bicomplex 
$\bigl(\Lambda[u_1,\ldots,u_m]\otimes\Z[\sk],d,d'\bigr)$ is zero, because the filtration in the spectral sequence of the bicomplex is infinite, and there are issues with the convergence, unless we consider an appropriate completion. For example, $1$ is a $(d'+d)$-cocycle representing a nontrivial total cohomology class. On the other hand $1$ becomes a coboundary if we consider infinite series:
\[
  (d'+d)(u_1-u_1v_1+u_1v_1^2-u_1v_1^3+\cdots)=1.
\]
This issue does not show up with $(R^*(\sK),d,d')$, because here the filtration is finite. If $\sK\ne\varDelta^{m-1}$, then any missing face $I\notin\sK$ gives rise to a polynomial $p$ with 
$(d'+d)p=1$. For example, if $I=\{1,2,\ldots,k\}$, then
\[
  (d'+d)(u_1-u_2v_1+u_3v_2v_1-\cdots+(-1)^k u_kv_{k-1}v_{k-2}\cdots v_1)=1.
\] 
\end{remark}

The functorial properties of the double cohomology~$\HH^*(\zk)$ follow from its definition and the appropriate properties of~$H^*(\zk)$ stated in Proposition~\ref{funcH}:

\begin{proposition}\label{functorialH}
Let $\mathcal L\subset\sK$ be simplicial complexes on the same vertex set~$[m]$. There are homomorphisms
\begin{itemize}
\item[(a)] $\bigl(\Lambda[u_1,\ldots,u_m]\otimes\Z[\sk],d,d'\bigr)\to  \bigl(\Lambda[u_1,\ldots,u_\ell]\otimes\Z[\mathcal L],d,d'\bigr),$
\item[(b)] $\bigl(R^*(\sK),d,d'\bigr)\to\bigl(R^*(\mathcal L),d,d'\bigr),$
\item[(c)] $\CH^\ast(\zk) \to \CH^\ast(\zl)$, 
\item[(d)] $\HH^*(\mathcal Z_{\mathcal K})\to\HH^*(\mathcal Z_{\mathcal L}).$
\end{itemize}
Furthermore, if $\sK_I$ is a full subcomplex for some $I\subset[m]$, then we have homomorphisms
\begin{itemize}
\item[(e)] $\bigl(\Lambda[u_i\colon i\in I]\otimes\Z[\sk_I],d,d'\bigr)\to  \bigl(\Lambda[u_1,\ldots,u_m]\otimes\Z[\mathcal K],d,d'\bigr),$
\item[(f)] $\bigl(R^*(\sK_I),d,d'\bigr)\to  \bigl(R^*(\mathcal K),d,d'\bigr),$
\item[(g)] $\CH^*(\mathcal Z_{\sk_I})\to\CH^*(\zk),$ 
\item[(h)] $\HH^*(\mathcal Z_{\sk_I})\to\HH^*(\zk).$
\end{itemize}
There are also homology versions of the homomorphisms, which map between $H_*$, $\CH_*$ and $\HH_*$ in the opposite direction.
\end{proposition}

\begin{remark}
If $\sK\subset\mathcal L$ is a subcomplex on a smaller vertex set $[\ell]$, $\ell<m$, then we do not get a map of bicomplexes $\bigl(\Lambda[u_1,\ldots,u_m]\otimes\Z[\sk],d,d'\bigr)\to
  \bigl(\Lambda[u_1,\ldots,u_\ell]\otimes\Z[\mathcal L],d,d'\bigr)$,
because $u_i\mapsto 0$ for $i\notin[\ell]$, but $d'u_i=1$. This can be fixed by considering ghost vertices.
\end{remark}

We finish this section with the following observation about Proposition \ref{functorialH}~(c) and its homology version. 
Using the Hochster decompositions of $H^\ast(\zk)$ and $H^\ast(\mc{Z}_\mc{L})$, we may write $f\colon \CH^\ast(\zk) \to \CH^\ast(\mc{Z}_\mc{L})$ as 
\begin{equation*}
\bigoplus_{p\geq 0, I\subset[m]} f_{I}^p \colon \bigoplus_{p\geq 0, I\subset[m]} \widetilde{H}^p(\sk_I) \to \bigoplus_{p\geq 0, I\subset[m]}  \widetilde{H}^p(\mc{L}_I), 
\end{equation*}
where $f_{I}^p \colon \widetilde{H}^p(\sk_I) \to \widetilde{H}^p(\mc{L}_I)$. Similarly, the map $g\colon \CH_\ast(\mc{Z}_\mc{L}) \to \CH_\ast(\zk)$ can also be decomposed into $\bigoplus_{p\geq 0, I\subset [m]}g_{p,I}$, where $g_{p, I}\colon \widetilde{H}_p(\mc{L}_I) \to \widetilde{H}_p(\sk_I).$ Hence, each of the two chain maps $f$ and $g$ can be discussed degree by degree as follows:
\begin{equation}\label{eq_functoriality_deg_by_deg}
\begin{tikzcd}
   \CH^p(\zk)\arrow[equal]{d} \arrow{r}{f^p} & \CH^p(\mc{Z}_\mc{L}) \arrow[equal]{d} \\
\displaystyle \bigoplus_{I\subset[m]}\widetilde{H}^p(\sk_I)\arrow{r}{\underset{I\subset[m]}{\bigoplus}f_I^p}&\displaystyle \bigoplus_{I\subset[m]}\widetilde{H}^p(\mc{L}_I)
\end{tikzcd}
\quad \text{and} \quad 
\begin{tikzcd}
   \CH_p(\mc{Z}_\mc{L})\arrow[equal]{d} \arrow{r}{g_p} & \CH_p(\zk) \arrow[equal]{d} \\
\displaystyle \bigoplus_{I\subset[m]}\widetilde{H}_p(\mc{L}_I)\arrow{r}{\underset{I\subset[m]}{\bigoplus} g_{p,I}}&\displaystyle \bigoplus_{I\subset[m]}\widetilde{H}_p(\mc{K}_I)
\end{tikzcd}
\end{equation}
respectively. We revisit \eqref{eq_functoriality_deg_by_deg} in Section~\ref{sec_mcycle}.

\section{Relation to the torus action}\label{relta}
Recall from~\cite[Propositions~3.1.5, 4.2.4]{bu-pa15} that the face ring $\Z[\sK]$ and the moment-angle complex $\zk$ are functorial with respect to simplicial maps $\sK\to\sK'$. 

Given simplicial complexes $\sK$ and $\sK'$ on the sets $[m]$ and $[n]$, respectively, a set map $\varphi\colon[m]\to [n]$ induces a \emph{simplicial map} $\varphi\colon\sK\to\sK'$ if $\varphi(I)\in\mathcal\sK'$ for any $I\in\sK$. For the corresponding face rings $\Z[\sK]=\Z[v_1,\ldots,v_m]/\mathcal I_\sK$ and $\Z[\sK']=\Z[v'_1,\ldots,v'_{n}]/\mathcal I_{\sK'}$, a simplicial map 
$\varphi\colon\sK\to\sK'$ induces a homomorphism
\begin{equation}\label{phi*}
  \varphi^*\colon\Z[\sK']\to\Z[\sK],\quad\varphi^*(v'_j)=\sum_{i\in\varphi^{-1}(j)}v_i.
\end{equation}

Furthermore, a simplicial map $\varphi\colon\sK\to\sK'$ induces a map $\varphi_{\mathcal Z}\colon\zk\to\mathcal Z_{\sK'}$. The latter is obtained by restricting to $\zk$ the map of polydiscs
\[
  (D^2)^m\to(D^2)^{n},\quad (z_1,\ldots,z_m)\mapsto(z'_1,\ldots,z'_{n}), 
\]
where
\[
  z'_j=\prod_{i\in\varphi^{-1}(j)}z_i,\quad\text{ for }j=1,\ldots,m,
\]
and $z'_j=1$ if $\varphi^{-1}(j)=\varnothing$. Note that we use the monoid structure on $D^2$ in this definition.

The isomorphisms of Theorem~\ref{zkcoh} are functorial with respect to the maps defined above. More precisely, a simplicial map $\varphi\colon\sK\to\sK'$ induces a morphism of differential graded algebras
\[
  \bigl(\Lambda[u'_1,\ldots,u'_{n}]\otimes\Z[\sK'],d\bigr)\to
  \bigl(\Lambda[u_1,\ldots,u_m]\otimes\Z[\sK],d\bigr),
\]
which is defined on the generators $v'_j$ by~\eqref{phi*} and on the $u'_j$ by the same formula. The induced homomorphism in cohomology coincides with $\varphi_{\mathcal Z}^*\colon H^*(\mathcal Z_{\sK'})\to H^*(\zk)$.

The torus action $T^m\times\zk\to\zk$ can be regarded as an example of the functoriality property described above. To see this, we need to bring ghost vertices into consideration. Recall that, for a simplicial complex $\sK$ on $[m]$, a \emph{ghost vertex} is a one-element subset $\{i\}\subset[m]$ such that $\{i\}\notin\sK$. Let $(\varnothing,[m])$ denote the simplicial complex on $[m]$ consisting of $\varnothing$ only (therefore having $m$ ghost vertices). The corresponding moment-angle complex is an $m$-torus: $\mathcal Z_{\varnothing,[m]}=T^m$.

Given $\sK$ on $[m]$, consider the simplicial complex $(\varnothing,[m'])\sqcup\sK$ on $[m']\sqcup[m]$ with $m=m'$, namely we add $m$ new ghost vertices labeled by $\{1', \dots, m'\}$. Then the simplicial map
$\varphi\colon(\varnothing,[m'])\sqcup\sK\to\sK$ identifying $i'$ with $i$ for each $i=1,\ldots,m$ gives rise to
$\mathcal Z_{(\varnothing,[m'])\sqcup\sK}=T^m\times\zk$ and $\varphi_{\mathcal Z}$ is the standard torus action map $T^m\times\zk\to\zk$.

Given a circle action $S^1\times X\to X$ on a space $X$, the induced map in cohomology has the form
\[
  H^*(X)\to H^*(S^1\times X)=\Lambda[u]\otimes H^*(X), \quad
  \alpha\mapsto 1\otimes\alpha+u\otimes\iota(\alpha) ,
\]
where $u\in H^1(S^1)$ is a generator and $\iota\colon H^*(X)\to H^{*-1}(X)$ is a derivation.

\begin{proposition}
The derivation corresponding to the $i^{th}$ coordinate circle action $S^1_i\times\zk\to\zk$ is induced by the derivation $\iota_i$ of the Koszul complex $(\Lambda[u_1,\ldots,u_m]\otimes\Z[\sK],d)$ given on the generators by
\[
  \iota_i(u_j)=\delta_{ij},\quad \iota_i(v_j)=0,\quad\text{ for } j=1,\ldots,m,
\]
where $\delta_{ij}$ is the Kronecker delta.

The derivation corresponding to the diagonal circle action $S^1_d\times\zk\to\zk$ coincides with the differential $d'$ given by~\eqref{d'uv}.
\end{proposition}

\begin{proof}
The $i^{th}$ coordinate circle action $S^1_i\times\zk\to\zk$ is the map of moment-angle complexes 
$\varphi_{\mathcal Z}\colon \mathcal Z_{(\varnothing,\{i'\})\sqcup\sK}\to\zk$ induced by the simplicial map 
$\varphi\colon(\varnothing,\{i'\})\sqcup\sK\to\sK$ sending the ghost vertex $i'$ to~$i$. The corresponding map of Koszul complexes is given by
\[
   \Lambda[u_1,\ldots,u_m]\otimes\Z[\sK]\to
   \Lambda[u'_i]\otimes\Lambda[u_1,\ldots,u_m]\otimes\Z[\sK],
\]
where $u_i\mapsto u'_i+u_i=1\otimes u_i+u'_i\otimes 1$, $u_j\mapsto 1\otimes u_j$ for $j\ne i$ and $v_j\mapsto 1\otimes v_j$ for any~$j$. Here $u'_i$ represents the generator of $H^1(S^1_i)$. This proves the first assertion. 

The second assertion follows from the fact that the derivation corresponding to the diagonal circle action is the sum of the derivations corresponding to the coordinate circle actions.
\end{proof}

\begin{remark}
The derivations $\iota_i$ were introduced and studied in the work of Amelotte and Briggs~\cite{am-br} under the name \emph{primary cohomology operations for $\zk$}. We expect that the higher cohomology operations from~\cite{am-br} are related to the differentials in the spectral sequence of the bicomplex 
$(\Lambda[u_1,\ldots,u_m]\otimes\Z[\sK],d,d')$ from the previous section. 
\end{remark}

\section{Techniques for computing $\HH^\ast(\zk)$}\label{sec_techniques}
In this section, we consider methods for the practical computation of $\HH^\ast(\zk)$. We show that $\HH^*(\zk)$ has rank one if and only if $\sK$ is a full simplex, and identify several classes of simplicial complexes $\sK$ for which $\HH^*(\zk)$ has rank two. We also show that the rank of $\HH^\ast(\zk)$ can be arbitrarily large.

\begin{proposition}\label{thm_main} 
For a simplicial complex $\sK$ on the vertex set $[m]$, the following conditions are equivalent:
\begin{itemize}
\item[(a)] all full subcomplexes of $\sK$ are acyclic;
\item[(b)] $\sK=\varDelta^{m-1}$ and $\zk=(D^2)^m$;
\item[(c)] $\zk$ is acyclic;
\item[(d)] $\HH^*(\zk) = \HH^{0,0}(\zk)= \Z$.
\end{itemize}
\end{proposition}

\begin{proof}
(a)$\Rightarrow$(b) A minimal non-face of $\sK$ is a full subcomplex that is not acyclic, as it is the boundary of a simplex. Hence, if all full subcomplexes are acyclic, then every subset of $[m]$ is a face of~$\sK$, so $\sK=\varDelta^{m-1}$.

\smallskip

(b)$\Rightarrow$(c) Obvious.

\smallskip

(c)$\Rightarrow$(a) This follows from Theorem~\ref{zkcoh}.

\smallskip

(b)$\Rightarrow$(d) Obvious.

\smallskip

(d)$\Rightarrow$(b) This follows from Corollary~\ref{cor_euler_char_of_HH_for_simplex}.
\end{proof}


Next is the basic example when the double cohomology has rank~$2$.

\begin{proposition}\label{prop_boundary_simplex}
Let $\sK=\partial\varDelta^{m-1}$, the boundary of an $(m-1)$-simplex. Then, 
\[
  \HH^{-k, 2\ell}(\zk)= \begin{cases} 
    \ZZ & \text{for } (-k, 2\ell)= (0,0),\, (-1, 2m);\\
    0 & \text{otherwise}. 
  \end{cases}
\]
\end{proposition}
\begin{proof}
Using the isomorphism $H^*(\zk)\cong H(R^*(\sK),d)$ (see Section \ref{sec_prel}), the only nontrivial cohomology groups are
\begin{align*}
  &H^0(\zk)=H^{0,0}(\zk)\cong\Z\langle[1]\rangle;\\
  &H^{2m-1}(\zk)=H^{-1,2m}(\zk)\cong\Z\langle[u_1v_2v_3\cdots v_m]\rangle,
\end{align*}
where $u_1v_2v_3\cdots v_m$ is a $d$-cocycle. Hence, the induced differential $d'$ on $H^*(\zk)\cong H(R^*(\sK),d)$ is trivial, so $\HH^*(\zk)$ is the same as $H^*(\zk)$.
\end{proof}

Recall the cochain complex $\CH^\ast(\zk)$ giving $\HH^{\ast}(\zk)$, see~\eqref{chcoho}. Let $\mathcal{K}\ast \mathcal{L}$ denote the simplicial join. Observe that $\mathcal{Z}_{\mathcal{K}\ast \mathcal{L}}\cong\zk\times\mathcal Z_{\mathcal L}.$ 
Then next result shows that the double cohomology can be arbitrarily large.

\begin{theorem}\label{thm_join}
For two simplicial complexes $\mathcal{K}$ and $\mathcal{L}$, if either  $H^*(\mathcal{Z}_{\mathcal{K}})$ or $H^\ast(\mathcal{Z}_{\mathcal{L}})$ is projective over the coefficients then there is an isomorphism of chain complexes 
\begin{equation}\label{eq_join_chain_complex}
\CH^\ast(\mathcal{Z}_{\mathcal{K}\ast \mathcal{L}}) \cong \CH^\ast(\mathcal{Z}_{\mathcal{K}}) \otimes \CH^\ast(\mathcal{Z}_{\mathcal{L}}).
\end{equation}
In particular, we have $\HH^\ast(\mathcal{Z}_{\mathcal{K}\ast \mathcal{L}}) \cong \HH^\ast(\mathcal{Z}_{\mathcal{K}}) \otimes \HH^\ast(\mathcal{Z}_{\mathcal{L}})$ with field coefficients. The same statement holds for the double homology $\HH_*$. 
\end{theorem}

\begin{proof}
The map 
\[
\phi\colon\CH^\ast(\zk) \oplus_{\CH^\ast(\mc{Z}_{\varnothing})}\CH^\ast(\mc{Z}_\mc{L})\rightarrow  \CH^\ast(\mc{Z}_{\sk\ast \mc{L}})
\]
induced by the Hochster decomposition \eqref{eq_Hochster_decomp} has a retraction given by projections onto the factors. The projections are cochain maps because, for every $x\not\in \sK$ and a subset $I$ of the set $V(\sk)$ of vertices in $\sK$, the map
\[
(\sK\ast \mc{L})_I\rightarrow (\mc{K}\ast \mc{L})_{I\cup \{ x\}}
\]
is the inclusion of a full subcomplex and similarly for $\mc{L}$. So we get that 
\begin{equation}\label{eq1}
  \CH^\ast(\mc{Z}_{\mc{K}\ast\mc{L}})\cong \CH^\ast(\zk) \oplus_{\CH^\ast(\mc{Z}_{\varnothing})}\CH^\ast(\mc{Z}_\mc{L})
  \oplus {\mathop\mathrm{coker} } \phi,
\end{equation}

where 
\begin{equation}\label{eq_join_coker}
\coker \phi=\bigoplus_{\substack{\varnothing \neq I \subset V(\sk),\\ \varnothing \neq J \subset V(\mc{L})}} \widetilde{H}^\ast((\sk\ast \mc{L})_{I\cup J}) \cong 
\bigoplus_{\substack{\varnothing \neq I \subset V(\sk),\\ \varnothing \neq J \subset V(\mc{L})}} \widetilde{H}^\ast(\sk_I) \otimes \widetilde{H}^\ast(\mc{L}_J).
\end{equation}
Here, the second isomorphism in \eqref{eq_join_coker} follows from the K\"unneth theorem together with the hypothesis on the projectiveness of $H^*(\mathcal{Z}_{\mathcal{K}})$ or $H^\ast(\mathcal{Z}_{\mathcal{L}})$. Therefore we have that 
\begin{equation}\label{eq2}
  \mathop\mathrm{coker} \phi \cong \overline{\CH}{}^*(\zk) \otimes \overline{\CH}{}^* (\mc{Z}_\mc{L}).
\end{equation}
where $\overline{\CH}{}^*(\zk)$ is the cokernel of $\CH^*(\mc{Z}_{\varnothing})\to \CH^*(\zk)$ and similarly for $\mathcal{L}$. Then isomorphisms \eqref{eq1} and \eqref{eq2} imply \eqref{eq_join_chain_complex}. The proof for $\HH_*$ is similar.
\end{proof}

The isomorphism of Theorem~\ref{thm_join} can be also deduced from the algebraic description of Theorem~\ref{1db}. We illustrate this by an example.

\begin{example}\label{ex_4gon}
Let $\sk_1=\sk_2=\partial \Delta^1$, two disjoint points. Then, $\sk=\sk_1 \ast \sk_2$ is the boundary of a square, $\zk\cong \mathcal{Z}_{\mathcal{K}_1} \times \mathcal{Z}_{\mathcal{K}_2} \cong  S^3\times S^3$ and $\HH^\ast(\mathcal{Z}_{\sk_i})=H^\ast(\mathcal{Z}_{\sk_i})$ for $i=1,2$ by Proposition \ref{prop_boundary_simplex}. 
The nontrivial cohomology groups are
\begin{align*}
  H^0(\zk)&=H^{0,0}(\zk)\cong\Z\langle1\rangle;\\ H^3(\zk)&=H^{-1,4}(\zk)\cong
  \Z\langle[u_1v_3],[u_2v_4]\rangle;\\
  H^6(\zk)&=H^{-2,8}(\zk)\cong\Z\langle[u_1u_2v_3v_4]\rangle,
\end{align*}
where the vertex sets of $\sk_1$ and $\sk_2$ are $\{1,3\}$ and $\{2,4\}$, respectively. Note that the induced differential $d'$ is trivial, which yields  $\HH^\ast(\zk)=H^\ast(\zk)$ and 
$\HH^\ast(\zk) \cong \HH^\ast(\mathcal{Z}_{\sk_1}) \otimes \HH^\ast(\mathcal{Z}_{\sk_2})$.
\end{example}

In the previous examples $\HH^*(\zk)$ behaved like $H^*(\zk)$. Here is an example of a major difference.

\begin{theorem}\label{thm_K_homotopy_discrete}
Let $\sk=\sk'\sqcup\mathit{pt}$ be the disjoint union of a nonempty simplicial complex $\sk'$ and a point. 
Then, 
\[
  \HH^{-k, 2\ell}(\mc{Z}_\mc{K})=\begin{cases} 
    \ZZ & \text{for }(-k,2\ell)=(0,0),~(-1,4);\\
    0 & \text{otherwise}.
  \end{cases}
\]
\end{theorem}

\begin{proof}
Assume $\sk'$ to be a simplicial complex on $[m]$, and let $\overline{[m]}\colonequals [m]\sqcup \{0\}$ be the vertex set of~$\sk$, with $0$ being the added disjoint vertex. Consider the cochain complex $\CH^\ast(\zk)=(H^{\ast}(\zk), d')$  giving $\HH^{\ast}(\zk)$.
We define a weight
\[
  w(J)\colonequals \sum_{j\in J}j
\]
for each subset $J\subset \overline{[m]}$ and consider the following increasing  filtration 
\begin{equation}\label{eq_filtration}
  \varnothing=F_{-1} \subset F_0 \subset \cdots \subset F_{p-1} \subset F_p \subset F_{p+1} \subset \cdots  \subset F_{\frac{m(m+1)}{2}}=H^\ast(\zk)
\end{equation}
of $H^{\ast}(\zk) =  \bigoplus_{J\subset \overline{[m]}} \widetilde{H}^\ast(\mathcal{K}_J)$, where 
\[
  F_p H^n(\zk) \colonequals   \bigoplus_{J\subset \overline{[m]} \colon  w(J)\leq p}
  \widetilde{H}^{n-|J|-1}(\mathcal{K}_J).
\]
The associated spectral sequence converges to $\HH^*(\zk)$ and has the $E_0$-page 
\[
  E_0^{p,q}=\left( F_pH^{p+q}(\zk)/F_{p-1}H^{p+q}(\zk)\right)=
  \bigoplus_{J\subset \overline{[m]} \colon  w(J)= p}
  \widetilde{H}^{p+q-|J|-1}(\mathcal{K}_J)
\]
with the differential $d_0\colon E_0^{p,q}\to E_0^{p,q-1}$ given by
\begin{equation}\label{d0dif}
  \bigoplus_{J\subset \overline{[m]} \colon  w(J)= p}
  \Bigl(\widetilde{H}^{p+q-|J|-1}(\mathcal{K}_J)\xrightarrow {d'}
  \widetilde{H}^{p+q-|J|-1}(\mathcal{K}_{J\setminus\{0\}})\Bigr).
\end{equation}

We now calculate the $E_1$-page. Note that the differential $d'$ in~\eqref{d0dif} is an isomorphism unless $0\in J$. For $w(J)=p=0$, we have  
\[
  E_1^{0,\ast}
  =\coker \bigl( \widetilde{H}^{-1}(\sk_{\{0\}}) \to \widetilde{H}^{-1}(\sk_\varnothing)\bigr) 
  \cong \ZZ =E_1^{0,0}.
\]
Now consider the case $w(J)=p\ge1$. Since $0$ is an isolated vertex in $\sk$, the differential $d'$ in~\eqref{d0dif} is an isomorphism unless $p+q-|J|-1=0$. Therefore,
\[
  E_1^{p,q}=\bigoplus_{\substack{J\subset\overline{[m]}\\w(J)=p}}
  \ker \Bigl(\widetilde{H}^{p+q-|J|-1}(\mathcal{K}_J)\xrightarrow {d'}
  \widetilde{H}^{p+q-|J|-1}(\mathcal{K}_{J\setminus\{0\}})\Bigr)
  =\bigoplus_{\substack{0\in J \subset \overline{[m]} \\|J|=p+q-1 \\ w(J)=p}}\ZZ_J, 
\]
where $\ZZ_J$ denotes the summand $\Z$ corresponding to a subset~$J$.
We rewrite the above as 
\[
  E_1^{p, \ast} = \bigoplus_{\substack{0\in J \subset \overline{[m]} \\ w(J)=p \\ 1\in J }}\ZZ_J \oplus      
  \bigoplus_{\substack{0\in J \subset \overline{[m]} \\ w(J)=p \\ 1\notin J}}\ZZ_J, \qquad p\ge1. 
\]
For $p\geq 2$, the differential $d_1 \colon E_1^{p,*} \to E_1^{p-1,*}$ maps $ \bigoplus_{\substack{0\in J \subset \overline{[m]} \\ w(J)=p \\ 1\in J }}\ZZ_J$ isomorphically to  $\bigoplus_{\substack{0\in J \subset \overline{[m]} \\ w(J)=p-1 \\ 1\notin J }}\ZZ_J$ and  maps $\bigoplus_{\substack{0\in J \subset \overline{[m]} \\ w(J)=p \\ 1\notin J}}\ZZ_J$ to $0$. Also, $d_1$ maps $E_1^{0,0}=\ZZ_\varnothing$ and $E_1^{1,2}=\ZZ_{\{0,1\}}$ to zero. 
See Figure \ref{fig_E_1_page}.
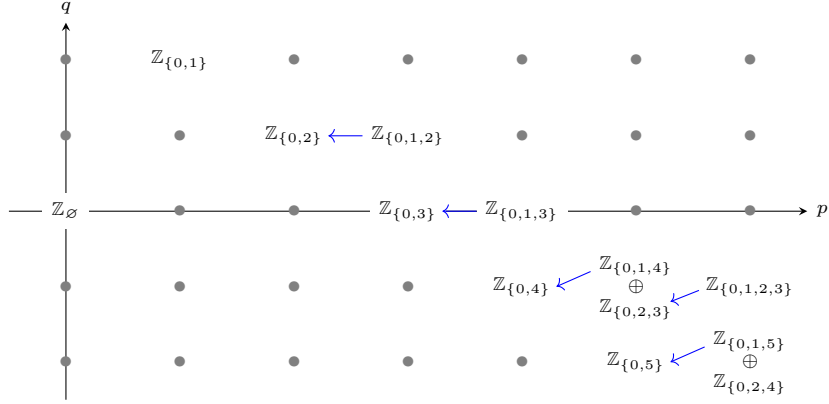
\begin{figure}
\begin{tikzpicture}[xscale=1.5]

\draw[-stealth] (0,-2.5)--(0,2.5); 
\node[above] at (0,2.5) {\scriptsize$q$};

\draw[-stealth] (-0.5,0)--(6.5,0);
\node[right] at (6.5,0) {\scriptsize$p$};

\foreach \x in {0,...,6}
	{\foreach \y in {-2, -1, 0, 1, 2}{
	\node at (\x, \y) {\textcolor{gray}{$\bullet$}};
	}}
\node[fill=white] at (0,0) {\scriptsize$\ZZ_\varnothing$};
\node[fill=white] at (1,2) {\scriptsize$\ZZ_{\{0,1\}}$};
\node[fill=white] at (2,1) {\scriptsize$\ZZ_{\{0,2\}}$};

\node[fill=white] at (3,1) {\scriptsize$\ZZ_{\{0,1,2\}}$};
\node[fill=white] at (3,0) {\scriptsize$\ZZ_{\{0,3\}}$};
\draw[blue, ->] (2.6,1)--(2.3,1);

\node[fill=white] at (4,0) {\scriptsize$\ZZ_{\{0,1,3\}}$};
\node[fill=white] at (4,-1) {\scriptsize$\ZZ_{\{0,4\}}$};
\draw[blue, ->] (3.6,0)--(3.3,0);

\node[fill=white] at (5,-1) {\scriptsize$\begin{array}{c} \ZZ_{\{0,1,4\}} \\ \oplus \\ \ZZ_{\{0,2,3\}} \end{array} $};
\node[fill=white] at (5,-2) {\scriptsize$\ZZ_{\{0,5\}}$};
\draw[blue, ->] (4.6,-0.8)--(4.3,-1);

\node[fill=white] at (6,-1) {\scriptsize$\ZZ_{\{0,1,2,3\}}$};
\node[fill=white] at (6,-2) {\scriptsize$\begin{array}{c} \ZZ_{\{0,1,5\}} \\ \oplus \\ \ZZ_{\{0,2,4\}} \end{array} $};
\draw[blue, ->] (5.55,-1.05)--(5.3,-1.2);
\draw[blue, ->] (5.6,-1.8)--(5.3,-2);
\end{tikzpicture}
\caption{The $E_1$-page and the differentials.}
\label{fig_E_1_page}
\end{figure}
Therefore, the spectral sequence collapses at $E_2$-page, and we have
\[
  E_2^{p,q}=E_{\infty}^{p,q}=
  \begin{cases} 
    \ZZ & (p,q)=(0,0) \text{ or } (1,2),\\
    0 & \text{otherwise},
  \end{cases}
\]
which completes the proof.  
\end{proof}

\begin{example}\label{ex_disjoint_pts}
Let $\sK$ be $m$ disjoint points. The nontrivial cohomology groups of $\zk$ are
\begin{align*}
  &H^0(\zk)=H^{0,0}(\zk)\cong\Z,\\
  &H^{q+1}(\zk)=H^{-q+1,2q}(\zk)
  \cong\Z^{(q-1)\binom mq},\quad 2\le q\le m,
\end{align*}
see~\cite[Example~4.7.6]{bu-pa15}. In fact, $\zk$ is homotopy equivalent to a wedge of spheres. The group $H^{-q+1,2q}(\zk)$ is generated by the cohomology classes $[u_Jv_i]$ with $|J|=q-1$ and $i\notin J$ subject to the relations arising from $du_L$ with $|L|=q$. The double cohomology groups 
$\HH^*(\zk)$ are therefore the cohomology groups of the cochain complex
\[
  0\longrightarrow H^{-m+1,2m}(\zk)\stackrel{d'}{\longrightarrow}
  \cdots\stackrel{d'}{\longrightarrow}
  H^{-2,6}(\zk)\stackrel{d'}{\longrightarrow} H^{-1,4}(\zk)\longrightarrow0
\]
together with $\HH^{0,0}(\zk)\cong\Z$. The differential $d'\colon H^{-2,6}(\zk)\to H^{-1,4}(\zk)$ is given by 
\[d'[u_iu_jv_k]=[u_jv_k]-[u_iv_k] \]
and therefore has cokernel of rank~1. Hence, $\HH^3(\zk)=\HH^{-1,4}(\zk)\cong\Z$.
By Theorem~\ref{thm_K_homotopy_discrete}, $\HH^n(\zk)=0$ for $n\ge4$.
Looking at the Betti numbers of this computation gives the binomial identity 
\[\sum_{q=2}^{m} (-1)^q(q-1)\binom mq=1.\]
\end{example}

Here is another example of a family of simplicial complexes with $\HH^*(\zk)$ of rank two.

\begin{theorem}\label{surgery}
Let $\sk=\sk' \cup_\sigma  \Delta^n$ be a simplicial complex obtained from a nonempty simplicial complex $\sk'$ by gluing an $n$-simplex along a proper, possibly empty, face $\sigma \in \sk'$. Then either $\sk$ is a simplex, or 
\[
  \HH^{-k, 2\ell}(\zk)=\begin{cases} 
    \ZZ & \text{for }(-k,2\ell)=(0,0),~(-1,4);\\
    0 & \text{otherwise}.
  \end{cases}
\]
\end{theorem}

\begin{proof}
Let $[m]$ be the vertex set of $\sk'$, and let $V(\sigma)\subset[m]$ be the vertex set of~$\sigma$. Let $|V(\sigma)|=s$. We may assume $s<m$, as otherwise $\sK=\Delta^n$ is a simplex. 

If $n=0$, then $\sk=\sk'\sqcup\mathit{pt}$ and we are in the situation of Theorem~\ref{thm_K_homotopy_discrete}.

We first consider the case when $s=n>0$, that is, $\sigma$ is a facet of $\Delta^n$. 
Let $m+1$ be the unique vertex of $\Delta^n$ not contained in~$\sigma$, so that $[m+1]$ is the vertex set of~$\sK$. 
Let $\mathcal{L}=\sk'_{[m]\setminus V(\sigma)}\sqcup\{m+1\}$. By Theorem~\ref{thm_K_homotopy_discrete}, the double cohomology of $\mathcal Z_{\mathcal L}$ has the required form. 
We claim that $\HH^*(\zk)\cong\HH^*(\mathcal Z_{\mathcal L})$. Indeed, consider the cochain complexes $\CH^\ast(\mathcal{Z}_{\sk})$ and $\CH^\ast(\mathcal{Z}_\mathcal{L})$. Since $\mathcal{L}$ is a full subcomplex of $\sk$, we use Proposition~\ref{functorialH} to get a cochain map 
\begin{equation}\label{eq_chain_map_cohom}
\CH^\ast(\mathcal{Z}_{\mathcal{L}})\rightarrow \CH^\ast(\mathcal{Z}_{\mathcal{K}}).
\end{equation}

Let $A$ be the cokernel of \eqref{eq_chain_map_cohom}. Then we have 
\[
A\cong \bigoplus_{J\in \mathcal{T}}\widetilde{H}^\ast(\sk_J),
\]
where $\mathcal{T}=\{J\in [m+1] \colon J\cap V(\sigma) \neq \varnothing\}$. We now calculate the cohomology of $A$ using a spectral sequence as in the proof of Theorem~\ref{thm_K_homotopy_discrete}. Here we use the weight function $w \colon \mathcal{T} \to \ZZ_{\geq 0}$ 
defined by $w(J)=|J\cap[m]|$; it counts the number of elements in $J\subset [m+1]$ after removing $m+1$. Then, we define 
\[
  F_p\colonequals \bigoplus_{J\in{\mathcal T}\colon w(J)\leq p} \widetilde H^\ast(\sK_J),
\]
which gives rise to a bounded increasing filtration on~$A$. 

The $E_0$-term of the associated spectral sequence breaks up into a direct sum of complexes 
\[
  \bigoplus_{J\in \mathcal{T} \colon m+1\in J} \left(\widetilde H^\ast(\sK_J) 
  \rightarrow \widetilde H^\ast(\sK_{J\setminus  \{m+1\}})\right).
\]
Since each inclusion $\sk_{J\setminus \{m+1\}} \to \sk_J$ is a homotopy equivalence, $E_0$ is an acyclic complex, which implies that $E_1$ is trivial. Hence, the chain complex $A$ is acyclic, so that the map  $\HH^\ast(\mathcal{L})  \to \HH^\ast(\mathcal{K})$ induced by \eqref{eq_chain_map_cohom} is an isomorphism, as claimed. 

Now consider the case where $s<n$. In this case, $\sK$ can be obtained from $\sk'$ by iterating the procedure of attaching a simplex along a facet:
\[
  \sK=\sk' \cup_{\sigma} \Delta^n = 
  \left( \left( \left(\sK'\cup_{\sigma} \Delta^s\right) \cup_{\Delta^s} \Delta^{s+1} \right) \cdots \cup_{\Delta^{n-1}} \Delta^n \right).
\]
Hence, the result follows by applying the argument above inductively.
\end{proof}

\begin{theorem}\label{thm_main2}
For a simplicial complex $\sK$, the following conditions are equivalent:
\begin{itemize}
\item[(a)] all full subcomplexes of $\sK$ are homotopy discrete sets of points;
\item[(b)] $\sK$ is flag and its one-skeleton $\mathrm{sk}^{1}(\sK)$ is a chordal graph;
\item[(c)] $\sK$ can be obtained by iterating the procedure of attaching a simplex along a (possibly empty) face, starting from a simplex. 
\end{itemize}
Each of the conditions above implies that $\sK$ is either a simplex, or 
\[
  \HH^{-k, 2\ell}(\mc{Z}_\mc{K})=\begin{cases} 
    \ZZ & \text{for } (-k,2\ell)=(0,0),~(-1,4);\\
    0 & \text{otherwise}.
  \end{cases}
\]
\end{theorem}

\begin{proof}
(a)$\Rightarrow$(b) If $\sK$ is not flag, then any minimal non-face with $\ge3$ vertices is a full subcomplex that is not homotopy discrete. Similarly, if $\mathrm{sk}^{1}(\sK)$ is not chordal, then any its chordless cycle with $\ge4$ vertices is a full subcomplex that is not homotopy discrete.

(b)$\Rightarrow$(c) This follows by considering a perfect elimination ordering on the vertex set of $\sK$, as in the proof of~\cite[Theorem 4.6]{GPTW}. 

(c)$\Rightarrow$(a) If $\sK$ is obtained as described in~(c), then each full subcomplex $\sK_I$ is obtained by the same procedure. This implies that each $\sK_I$ is a disjoint union of contractible spaces.

The formula for the double cohomology follows from the description in~(c) and Theorem~\ref{surgery}.
\end{proof}

\begin{remark}
Note that the class of simplicial complexes described in Theorem~\ref{surgery} is strictly larger than that described in Theorem~\ref{thm_main2}: in the former case one starts with an arbitrary simplicial complex~$\sK'$, while in the latter one starts with a simplex. Also, the class described in Theorem~\ref{surgery} does not exhaust all examples with double cohomology of rank $2$ in bidegrees $(0,0)$ and $(-1,4)$, as we illustrate in Subsection~\ref{fexam}.
\end{remark}

\begin{remark}
A simplicial complex $\sK$ is called \emph{Golod} (over a field $\mathbf{k}$) if the multiplication and all higher Massey products in $\Tor_{\mathbf{k}[v_1, \dots, v_m]}(\mathbf{k}[\sK], \mathbf{k})$ are trivial.
If~$\sK$ is flag and $\mathrm{sk}^{1}(\sK)$ is a chordal graph, then $\sK$ is a \emph{Golod complex} and $\zk$ is homotopy equivalent to a wedge of spheres by~\cite[Theorem~4.6]{GPTW}. Conversely, if $\sK$ is a Golod complex, then $\mathop{\mathrm{sk}}^1(\sK)$ is a chordal graph, but $\sK$ may fail to be flag. Also, if $\mathop{\mathrm{sk}}^1(\sK)$ is a chordal graph, then $\zk$ may fail to be homotopy equivalent to a wedge of spheres, and $\sK$ may fail to be Golod. An example of this situation is provided by a minimal (6-vertex) triangulation of~$\mathbb R P^2$, see Example~\ref{rp2ex} and~\cite[Example~3.3]{GPTW}.
\end{remark}

\begin{remark}
The condition that $\zk$ is homotopy equivalent to a wedge of spheres does not imply the same condition for $\sK$ itself. An example of this situation is provided by a minimal (7-vertex) triangulation of a $2$-torus $T^2$, see~\cite[Proposition~2.2]{Li15}.
By~\cite[Theorem 1.3]{IK20}, if $\sK$ is a triangulation of a closed connected surface, then $\sK$ is Golod if and only if it is 1-neighborly, that is, any two vertices in $\sK$ are connected by an edge. Clearly, such $\sK$ cannot be flag.
\end{remark}

\section{The case of an $m$-cycle}\label{sec_mcycle}
Here we apply the technique developed in the previous section to calculate the double cohomology of the moment-angle complex $\mathcal Z_{\mathcal L}$ corresponding to an $m$-cycle~$\mathcal L$. By a result of McGavran~\cite{mcga79}, the moment-angle complex $\mathcal Z_{\mathcal L}$ is homeomorphic to a connected sum of sphere products:
\[
  \mathcal Z_{\mathcal L}\cong \mathop{\#}_{k=3}^{m-1} \bigl(S^k\times S^{m+2-k}\bigr)^{\#(k-2)\binom{m-2}{k-1}}.
\]

\begin{example}\label{ex_boundary_mgon}
Let $\mathcal L$ be a $5$-cycle (the boundary of a pentagon) with the vertices numbered counterclockwise. 
Then we have $\mathcal Z_{\mathcal L}\cong(S^3\times S^4)^{\# 5}$. The nontrivial cohomology groups and their generators in the Koszul dga~\eqref{eq_Koszul_cohom} are given by
\begin{align*}
  H^0(\mathcal Z_{\mathcal L})&=H^{0,0}(\zk)\cong\Z\langle1\rangle,\\ 
  H^3(\mathcal Z_{\mathcal L})&=H^{-1,4}(\zk)\cong\Z\langle[u_1v_3],[u_1v_4],[u_2v_4],[u_2v_5],[u_3v_5]\rangle,\\
  H^4(\mathcal Z_{\mathcal L})&=H^{-2,6}(\zk)\cong\Z\langle[u_4u_5v_2],[u_2u_3v_5],[u_5u_1v_3],
  [u_3u_4v_1],[u_1u_2v_4]\rangle,\\
  H^7(\mathcal Z_{\mathcal L})&=H^{-3,10}(\zk)\cong\Z\langle[u_1u_2u_3v_4v_5]\rangle,
\end{align*}
see \cite[Example~4.6.11]{bu-pa15}.
There is only one nontrivial differential $d'$:
\[
  0\longrightarrow H^{-2,6}(\mathcal Z_{\mathcal L})\stackrel{d'}{\longrightarrow} H^{-1,4}(\mathcal Z_{\mathcal L})\longrightarrow0.
\]
It is given on the basis elements by
\begin{align*}
  d'[u_4u_5v_2]&=[u_5v_2]-[u_4v_2]=[u_2v_5]-[u_2v_4],\\
  d'[u_2u_3v_5]&=[u_3v_5]-[u_2v_5],\\
  d'[u_5u_1v_3]&=[u_1v_3]-[u_5v_3]=[u_1v_3]-[u_3v_5],\\
  d'[u_3u_4v_1]&=[u_4v_1]-[u_3v_1]=[u_1v_4]-[u_1v_3],\\
  d'[u_1u_2v_4]&=[u_2v_4]-[u_1v_4].
\end{align*}
The corresponding matrix
\[
\begin{pmatrix}
0& 0& 1& -1& 0\\
0& 0& 0& 1& -1\\
-1& 0&0& 0& 1\\
1&-1& 0& 0& 0\\
0& 1&-1& 0& 0
\end{pmatrix}
\]
has rank $4$ and defines a homomorphism onto a direct summand of~$\Z^5$. It follows that the nontrivial double cohomology groups are
\[
  \HH^{0,0}(\mathcal Z_{\mathcal L})\cong\HH^{-1,4}(\mathcal Z_{\mathcal L})\cong\HH^{-2,6}(\mathcal Z_{\mathcal L})\cong\HH^{-3,10}(\mathcal Z_{\mathcal L})\cong\Z.
\]
\end{example}

The next result extends this calculation to an arbitrary $m$-cycle. 

\begin{theorem}\label{prop_HH_mgon}
Let $\mc{L}$ be an $m$-cycle for $m\geq 5$. Then, the double cohomology of $\mc{Z}_\mc{L}$ is 
\[
\HH^{-k, 2\ell}(\mc{Z}_\mc{L})=
\begin{cases} \Z, & (-k, 2\ell)=(0,0), (-1,4), (-m+3, 2(m-2)), (-m+2, 2m);\\
0, & \text{otherwise}. \end{cases}
\]
\end{theorem}
\begin{proof}
Let $\sk$ be the simplicial complex obtained  by adding a 2-simplex $\{1,2,3\}$ to $\mc{L}$ and consider the natural inclusion $\mc{L}\hookrightarrow \sk$ (see Figure \ref{fig_emb_mgon}). One can apply Theorem \ref{surgery} for $\sk$, which gives us  
\begin{equation}\label{eq_mgon_HHzk}
\HH^{-k, 2\ell}(\zk)=\begin{cases} \Z, & (-k, 2\ell)=(0,0), (-1,4); \\
0, & \text{otherwise}.
\end{cases}
\end{equation}

\begin{figure}
\begin{tikzpicture}[scale=0.7]
	\fill [black] (0,0) circle (2pt) node[above left] at (0,0){\small 1};
	\fill [black] (1/2,1) circle (2pt) node[above left] at (1/2,1){\small 2};
    	\fill [black] (3/2,3/2) circle (2pt) node[above left] at (3/2,3/2){\small 3};
    	\fill [black] (2.5,3/2) circle (2pt) node[above] at (2.5,3/2){\small 4};
    	\fill [black] (0,-1) circle (2pt) node[left] at (0,-1){\small$m$};
	
	\node[rotate=-15] at (1, -1.7) {$\cdots$};
	\node[rotate=-60]  at (3.3,0.5) {$\cdots$};
	
	\node at (1.5, 0) {$\mathcal{L}$};
	\node at (5,0) {$\hookrightarrow$};
	\draw (1/2,-1.5)--(0,-1)--(0,0)--(1/2,1)--(3/2,3/2)--(2.5, 1.5)--(3, 1);
	
\begin{scope}[xshift=200]
	\draw[fill=yellow] (0,0)--(1/2,1)--(3/2,3/2)--cycle;
	\fill [black] (0,0) circle (2pt) node[above left] at (0,0){\small 1};
	\fill [black] (1/2,1) circle (2pt) node[above left] at (1/2,1){\small 2};
    	\fill [black] (3/2,3/2) circle (2pt) node[above left] at (3/2,3/2){\small 3};
    	\fill [black] (2.5,3/2) circle (2pt) node[above] at (2.5,3/2){\small 4};
    	\fill [black] (0,-1) circle (2pt) node[left] at (0,-1){\small$m$};
	
	\node[rotate=-15] at (1, -1.7) {$\cdots$};
	\node[rotate=-60]  at (3.3,0.5) {$\cdots$};
	\node at (1.5, 0) {$\mathcal{K}$};

	\draw (1/2,-1.5)--(0,-1)--(0,0)--(1/2,1)--(3/2,3/2)--(2.5, 1.5)--(3, 1);
\end{scope}
\end{tikzpicture}
\caption{An inclusion of $m$-gon to a simplicial complex.}
\label{fig_emb_mgon}
\end{figure}
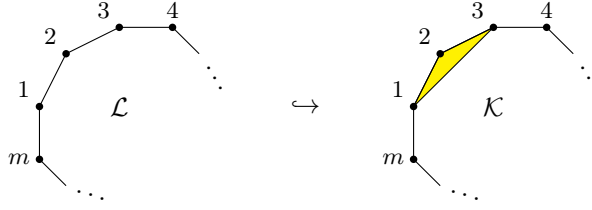

Now, we apply the functoriality property of Proposition~\ref{functorialH} to get chain maps
\begin{equation}\label{eq_mgon_functoriality}
f \colon \CH^\ast(\zk) \to \CH^\ast(\mc{Z}_\mc{L}) \quad \text{and} \quad  g\colon \CH_\ast(\mc{Z}_\mc{L})\to \CH_\ast(\zk).
\end{equation}
Recall from \eqref{eq_functoriality_deg_by_deg} that the maps $f$ and $g$ in \eqref{eq_mgon_functoriality} can be decomposed into $\bigoplus_{p\geq 0}f^p$ and $\bigoplus_{p\geq 0}g_p$ respectively. In this case,  we only need to consider the cases $p=0$ and $1$ because 
$\widetilde{H}^p(\sk_I)$ and $\widetilde{H}^p(\mathcal L_I)$ vanish for $p>1$.

First we consider the case $p=1$. Note that $\widetilde{H}^1(\mc{L}_I)=\Z$ only when  $I=[m]$. In this case, $\widetilde{H}^1(\mc{K}_{[m]})=\Z$ as well, hence $\coker f^1=0$. Thus $f^1$ yields a short exact sequence 
\begin{equation}\label{eq_mgon_ses_p=1}
\begin{tikzcd}[column sep=small]
0\rar& \ker f^1 \rar&  \CH^1(\zk)\rar{f^1} & \CH^1(\mc{Z}_\mc{L}) \rar& 0.
\end{tikzcd}
\end{equation}

Observe that $\widetilde{H}^1(\sk_I)=\Z$ if $I=[m]$ or $[m]\setminus \{2\}$. In particular, when $I=[m]$, the map $f_{[m]}^1 \colon \widetilde{H}^1(\sk_{[m]}) \to \widetilde{H}^1(\mc{L}_{[m]})$ is an isomorphism. Hence the chain complex $\ker f^1$ with the differential $d'$ induced from that of $\CH^\ast(\zk)$ consists of a single nontrivial term $\ker \big( f^1_{[m]\setminus\{2\}}\colon \widetilde{H}^1(\sk_{[m]\setminus \{2\}}) \to \widetilde{H}^1(\mc{L}_{[m]\setminus\{2\}})\big)$ which is of bidegree $(-m+3, 2(m-1))$. This implies that 
\begin{equation}\label{eq_mgon_Hker}
H^{-k, 2\ell}(\ker f^1, d')=\begin{cases} \Z, & (-k, 2\ell)=(-m+3, 2(m-1)); \\
0, &\text{otherwise}.
\end{cases}
\end{equation}

Now, we consider the long exact sequence of double cohomology
\begin{equation*}
\begin{tikzcd}[row sep=tiny, column sep=tiny]
\qquad \qquad\qquad \qquad \cdots \rar & \HH^{-\ell+2, 2\ell}(\zk)\rar  & H^{-\ell+2, 2\ell}(\mc{Z}_\mc{L}) \rar &~  \\
H^{-\ell+3, 2(\ell-1)}(\ker f^1, d') \rar &\HH^{-\ell+3, 2(\ell-1)}(\mc{Z}_\mc{K}) \rar &\HH^{-\ell+3, 2(\ell-1)}(\mc{Z}_\mc{L}) \rar & \cdots  
\end{tikzcd}
\end{equation*}
induced from the short exact sequence \eqref{eq_mgon_ses_p=1}. Then we conclude by \eqref{eq_mgon_HHzk} and \eqref{eq_mgon_Hker} that 
\begin{equation}\label{eq_mgon_HH_m}
\HH^{-\ell+2, 2\ell}(\mc{Z}_\mc{L})=
\begin{cases} 
H^{-m+3, 2(m-1)}(\ker f^1, d')=\Z, & \ell=m;\\
0, &\text{otherwise}.
\end{cases}
\end{equation}

For the case $p=0$, we consider the chain map $g_0\colon \CH_0(\mc{Z}_\mc{L}) \to \CH_0(\zk)$. We have $\widetilde H_0(\mc{K}_I)=\widetilde H_0(\mc{L}_I)$ unless $I$ satisfies $\{1,3\} \subset I \subsetneq [m]\setminus \{2\}$. In the latter case,
\[
\rank \widetilde{H}_0(\mc{L}_I) = \rank \widetilde{H}_0(\sk_I) +1.
\]
Hence, $\coker g_{0,I}$ is trivial for all $I\subset [m]$ and $\ker g_{0,I}$ is of rank $1$ only when  $\{1,3\} \subset I \subsetneq [m]\setminus \{2\}$. Therefore, $g_0$ gives a short exact sequence of chain complexes 
\begin{equation}\label{eq_mgon_ses_p=0}
\begin{tikzcd}[column sep=small]
0 \rar& \ker g_0 \rar & \CH_0(\mc{Z}_\mc{L}) \rar{g_0}&  \CH_0(\zk) \rar&  0,
\end{tikzcd}
\end{equation}
where the chain complex $\ker g_0$ is given by 
\begin{multline}\label{eq_mgon_ker_g0}
0 \to \ker g_{0,\{1,3\}} \stackrel{\partial'}\longrightarrow  
\displaystyle\bigoplus_{\small{\substack{\{1,3\} \subset I  \subsetneq [m]\setminus \{2\} \\ |I| =  3}} } \ker g_{0,I}
\stackrel{\partial'}\longrightarrow \cdots\\
\cdots \stackrel{\partial'}\longrightarrow 
\displaystyle\bigoplus_{\small{\substack{\{1,3\} \subset  I \subsetneq  [m]\setminus \{2\} \\ |I| =  m-3}} } \ker g_{0,I} 
\stackrel{\partial'}\longrightarrow
\displaystyle\bigoplus_{\small{\substack{\{1,3\} \subset  I\subsetneq [m]\setminus \{2\} \\ |I| =  m-2} }} \ker g_{0,I} 
\to 0.
\end{multline}
Here the bigrading of $\bigoplus_{\small{\substack{\{1,3\} \subset  I\subsetneq [m]\setminus \{2\} \\ |I| =  \ell} }} \ker g_{0,I}$ is $(-\ell+1,2\ell)$ and the total grading is $\ell+1$. The boundary operator 
induced from that of $\CH_\ast(\mc{Z}_\mc{L})$ is given by
\begin{equation}\label{eq_boundary'_at_0}
\partial' = -\sum_{
\substack{\{1,3\} \subset I\subsetneq  [m]\setminus\{2\},\\
j\in [m]\setminus I}} \varepsilon(j,I)\phi_{0;I,j}
\end{equation}
where $\phi_{0;I,j} \colon \widetilde{H}_0(\mc{L}_I) \to \widetilde{H}_0(\mc{L}_{I\cup \{j\}})$ is the homomorphism induced by the inclusion $\mc{L}_I\hookrightarrow \mc{L}_{I\cup \{j\}}$, see Subsection~\ref{homology_HH}. To be more precise, if $v_i$ is the $i$th vertex of $\mc{L}$, then $\ker g_{0,I}$ for $\{1,3\} \subset I \subsetneq [m]\setminus \{2\}$ is generated by 
\[
  x_I \colonequals [v_1]-[v_3]\in \widetilde{H}_0(\mc{L}_I)
  \subset H_{-|I|+1, 2|I|}(\zl)
\]
and $\phi_{0;I,j}$ sends $x_I$ to $x_{I\cup \{j\}}$. 

Now, we consider the augmented simplicial cochain complex $\mathscr{C}^\ast$
\[
0\to C^{-1} \to  C^{0} \to \cdots \to C^{m-6} \to C^{m-5} \to 0
\]
of the boundary $\partial \Delta^{m-4}$ of the $(m-4)$-simplex on $[m]\setminus \{1,2,3\}$. For each simplex $I\subsetneq [m]\setminus \{1,2,3\}$, we denote by $y_I$ the corresponding generator of the simplicial cochain group $C^{|I|-1}$. Then, the chain map $\theta \colon \ker g_0 \to \mathscr{C}^\ast[4]$  defined by $\theta(x_I)=-y_I$ is an isomorphism of chain complexes, where the grading of $\mathscr{C}^\ast$ is shifted by $4$ and the minus sign matches the minus sign in \eqref{eq_boundary'_at_0}. It follows that
\[
  H_{-k, 2\ell}(\ker g_0, \partial')=
  \begin{cases} \mathbb Z,& (-k, 2\ell)=(-m+3, 2(m-2));\\
  0,&\text{otherwise}.
  \end{cases}
\]

Consider the homology long exact sequence of \eqref{eq_mgon_ses_p=0}:
\[
\begin{tikzcd}[row sep=tiny, column sep=tiny]
 &\qquad \qquad \cdots \rar  & \HH_{-\ell+2, 2(\ell-1)}(\zk) \rar &~  \\
H_{-\ell+1, 2\ell}(\ker g_0, \partial') \rar &\HH_{-\ell+1, 2\ell}(\mc{Z}_\mc{L}) \rar &\HH_{-\ell+1, 2\ell}(\zk) \rar & \cdots.
\end{tikzcd}
\]
Assuming $m\geq 6$, we have 
\begin{equation}\label{eq_mgon_HH_m-1}
\HH_{-\ell+1, 2\ell}(\mc{Z}_\mc{L})\cong 
\begin{cases} H_{-m+3, 2(m-2)}(\ker g_0, \partial')=\Z, &\ell=m-2;\\
\HH_{-1,4}(\zk)=\Z, & \ell=2;\\
0, & \text{otherwise}.
\end{cases} 
\end{equation}
The double cohomology has the same form by the universal coefficient theorem, as $\HH_\ast(\zl)$ is free. The result follows from \eqref{eq_mgon_HH_m} and \eqref{eq_mgon_HH_m-1} together with the obvious isomorphism $\HH^{0,0}(\zl)=\Z.$ For $m=5$, we refer to Example~\ref{ex_boundary_mgon}. 
\end{proof}

\begin{remark}
When $\mc{L}$ is a $4$-cycle, we have $\HH^{-1,4}(\zl)\cong \Z^2$ by Example \ref{ex_4gon}. Hence, Proposition~\ref{prop_HH_mgon} holds for the case of $m=4$ as well, by counting the degree $(-1,4)$ twice. 
\end{remark}

\section{Further observations, examples and questions}\label{sec_examples}

\subsection{Top classes, wedge decomposability and duality}\label{subsec_top_class}

An element $\alpha\ne0$ in  $\widetilde{H}^*(\sK)$ is called a \emph{top class} if the restriction of $\alpha$ to $\widetilde H^*(\sK_I)$ is~$0$ for any proper $I\subset [m]$. For example, if $\sK$ is a triangulated connected closed $(n-1)$-dimensional (pseudo)manifold, then the fundamental class is a top class in $\widetilde H^{n-1}(\sK)$.

\begin{proposition}\label{topclass}
A top class $\alpha\in \widetilde{H}^{n-1}(\sK)$ survives to $\HH^{-(m-n),2m}(\zk)$.
\end{proposition}

\begin{proof}
This follows from the geometric definition of $d'$ (Subsection~\ref{subsec_HH_Cohomology}).
\end{proof}

We call a simplicial complex $\sK$ \emph{wedge decomposable} if it can be written as a nontrivial union $\mathcal L\cup_{\varDelta^t}\mathcal M$ of two simplicial complexes $\mathcal L$ and $\mathcal M$ along a nonempty simplex $\varDelta^t$ that is not the whole of $\sK$.
This corresponds to a graph being a nontrivial clique sum.
We have shown in Theorem \ref{surgery} that if $\mathcal L$ or $\mathcal M$ is a simplex, then $\HH^*(\zk)\cong\Z\oplus\Z$ with one $\mathbb{Z}$ summand in each of the bidegrees  $(0,0)$ and $(-1,4)$. Examples~\ref{TwoTriangles} and~\ref{TwoSquares} below give two more wedge decomposable complexes  having this property, in which neither $\mathcal{L}$ nor $\mathcal{M}$ is a simplex.

\begin{question}
Is it true for all wedge decomposable $\sK$ to have $\HH^*(\zk)\cong\Z\oplus\Z$ in bidegrees $(0,0)$ and $(-1,4)$? 
Does there exist a non-wedge-decomposable complex $\sK$ also having this property?
\end{question}

Recently Valenzuela \cite{val23} has shown all wedge decomposable $\sK$ have $\HH^*(\zk)\cong\Z\oplus\Z$ in bidegrees $(0,0)$ and $(-1,4)$ giving an affirmative answer to the first question. 

As an immediate consequence of Proposition \ref{topclass}, we have that if $\sk(\neq \partial\varDelta^1)$  has $\HH^*(\zk)\cong\Z\oplus\Z$ in bidegrees $(0,0)$ and $(-1,4)$, then $\sk$ does not have a top class. For instance, we know from Proposition~\ref{prop_boundary_simplex} that $\HH^*(\zk)\cong\Z\oplus \Z$ in bidegrees $(0,0)$ and $(-1, 2m)$ when $\sK$ is the boundary of a simplex $\partial \varDelta^{m-1}$. In particular, $\HH^{-1, 2m}(\zk) \cong \Z$ is generated by the class represented by the top class in $H^{2m-1}(\sk)$. We also have the following. 

\begin{proposition}
If $\sK$ has a top class, then $\sK$ is not wedge decomposable.
\end{proposition}

\begin{proof}
Follows from the Mayer--Vietoris sequence. 
\end{proof}

If $\sK$ is a triangulated $(n-1)$-dimensional sphere, then a generator $\alpha\in\widetilde H^{n-1}(\sK)\cong\Z$ is a top class and, therefore,
\begin{equation}\label{Zgoren}
  \HH^{-(m-n),2m}(\zk)\cong\widetilde H^{n-1}(\sK)\cong\Z.
\end{equation}

A simplicial complex $\sK$ of dimension $(n-1)$ is called a \emph{Gorenstein complex} over a field $\mathbb F$ if the face ring $\mathbb F[\sK]$ is Gorenstein, that is, 
\[
  \Tor^{-i,*}_{\mathbb F[v_1,\ldots,v_m]}\bigl(\mathbb F[\sK],\mathbb F\bigr)=0 \text{ for }i>m-n
  \text{ and } \Tor^{-(m-n),*}_{\mathbb F[v_1,\ldots,v_m]}
  \bigl(\mathbb F[\sK],\mathbb F\bigr)\cong\mathbb F.
\]
A Gorenstein complex $\sK$ is called \emph{Gorenstein*} if $\sK$ is not a cone over a subcomplex. If $\sK$ is Gorenstein, then $\sK=\mathcal L\mathbin{*}\varDelta^s$ for some $s$, where $\mathcal L$ is Gorenstein* and $\varDelta^s$ is a simplex, see~\cite[\S3.4]{bu-pa15}. 

There is the following homological characterisation of Gorenstein* complexes: $\sK$ is Gorenstein* if and only if the link $\lk_\sk I$ of any simplex $I\in\sk$, including~$\varnothing$, is a homology sphere of dimension $\dim\lk_\sk$, cf.~\cite[Theorem~3.4.2]{bu-pa15}. It follows that a triangulated sphere $\sK$ is Gorenstein* over any~$\mathbb F$.
A Gorenstein* complex~$\sK$ satisfies~\eqref{Zgoren} with coefficients in~$\mathbb F$.

A graded commutative connected $\mathbb F$-algebra $A$ is called a \emph{Poincar\'e algebra} of dimension $d$ if $A=\bigoplus_{i=0}^dA^i$, the graded components $A^i$ are finite dimensional over~$\mathbb F$, and the $\mathbb F$-linear maps
\[
  A^i\to\Hom_{\mathbb F}(A^{d-i},A^d),\quad
  a\mapsto\phi_a,\quad\text{where } \phi_a(b)=ab,
\]
are isomorphisms for $0\le i\le d$.

By~\cite[Theorem~4.6.8]{bu-pa15}, the ordinary cohomology $H^*(\zk;\mathbb F)$ is a Poincar\'e algebra if and only if $\sK$ is a Gorenstein complex. If $\sK$ is Gorenstein* of dimension $(n-1)$, then $\Tor^{-(m-n),2m}_{\mathbb F[v_1,\ldots,v_m]}(\mathbb F[\sK],\mathbb F)\cong\mathbb F$ and $H^*(\zk;\mathbb F)$ is a Poincar\'e algebra of dimension $(m+n)$. For the double cohomology $\HH^*(\zk)$, we have the following.

\begin{proposition}\label{HHPA}
If $\sk$ is a Gorenstein complex over a field $\mathbb F$, then the double cohomology $\HH^\ast(\zk;\mathbb F)$ is a Poincar\'e algebra. In particular, if $\sk$ is Gorenstein*  of dimension~$(n-1)$, then
\[
  \dim \HH^{-k, 2\ell}(\zk;\mathbb F) = \dim \HH^{-(m-n)+k, 2(m-\ell)}(\zk;\mathbb F). 
\]
\end{proposition}

\begin{proof}
If $\sK=\mathcal L\mathbin{*}\varDelta^s$, where $\mathcal L$ is Gorenstein*, then we have $\HH^*(\zk)=\HH^*(\mathcal Z_{\mathcal L})$ by Proposition~\ref{thm_main} and Theorem~\ref{thm_join}. Hence, we can assume that $\sK$ is Gorenstein* of dimension~$(n-1)$. Then we have
\[
  H^{m+n}(\zk;\mathbb F)\cong
  \Tor^{-(m-n),2m}_{\mathbb F[v_1,\ldots,v_m]}\bigl(\mathbb F[\sK],\mathbb F\bigr)\cong\widetilde H^{n-1}(\sK;\mathbb F)
  \cong \mathbb F,
\]
generated by the top class of $\widetilde H^{n-1}(\sK;\mathbb F)$ or by the cohomology class of the monomial $u_{[m]\setminus I}v_I$ in the Koszul complex, for any $(n-1)$-simplex $I\in\sK$.
It follows from Proposition \ref{topclass} that 
the top class of $\widetilde H^{n-1}(\sK;\mathbb F)$ is a $d'$-cocycle which is not a coboundary. Since $d'$ satisfies the Leibniz formula with respect to the product in $H^*(\zk;\mathbb F)$ by Theorem~\ref{CHDGA} and $H^*(\zk;\mathbb F)$ is a Poincar\'e algebra, $\CH^\ast(\zk;\mathbb F)= (H^\ast(\zk;\mathbb F), d')$ is a commutative differential graded algebra with Poincar\'e duality in the sense of~\cite{la-st04}. Then $\HH^*(\zk;\mathbb F)$ is a Poincar\'e algebra by~\cite[Proposition~4.7]{la-st04} (the argument is given there with $\mathbb Q$-coefficients, but it is the same for any field~$\mathbb F$).
\end{proof}

The converse of Proposition~\ref{HHPA} does not hold, 
unlike the situation with the ordinary cohomology $H^*(\zk)$. For example, if $\sK$ is $m$ disjoint points, then $\HH^*(\zk)$ is a Poincar\'e algebra by Theorem~\ref{thm_K_homotopy_discrete}, but $\sK$ is not Gorenstein if $m>2$. We may ask the following.

\begin{question}\label{question_duality}
Give a homological characterisation of simplicial complexes $\sK$ for which the double cohomology $\HH^*(\zk)$ is a Poincar\'e algebra.
\end{question}

\subsection{Examples}\label{fexam}
Here is an example of a complex obtained by the procedure described in Theorem~\ref{surgery}, but not in the class described in Theorem~\ref{thm_main2}.

\begin{example}\label{SquareEdge}
Let $\sK$ be a $4$-cycle $(1,2,3,4)$ with an edge $(4,5)$ attached at the vertex~$\{4\}$. 
The nontrivial cohomology groups are
\begin{align*}
  H^0(\zk)&=H^{0,0}(\zk)\cong\Z\langle1\rangle,\\ 
  H^3(\zk)&=H^{-1,4}(\zk)\cong\Z\langle[u_1v_3],[u_2v_4],[u_1v_5],[u_2v_5],[u_3v_5]\rangle,\\
  H^4(\zk)&=H^{-2,6}(\zk)\cong\Z\langle[u_1u_2v_5],[u_2u_3v_5],[u_1u_3v_5],[u_3u_5v_1],[u_4u_5v_2]\rangle,\\
  H^{5}(\zk)&=H^{-3,8}(\zk)\cong\Z\langle[u_1u_2u_3v_5]\rangle,\\
  H^{6}(\zk)&=H^{-2,8}(\zk)\cong\Z\langle[u_1u_2v_3v_4]\rangle,\\
  H^7(\zk)&=H^{-3,10}(\zk)\cong\Z\langle[u_3u_4u_5v_1v_2]\rangle.
\end{align*} 
The nontrivial parts of the cochain complex $\CH^\ast(\zk)$ are 
\begin{align}
\label{eq_CH2_squareedge}
  &0\longrightarrow H^{-3,10}(\zk)\stackrel{d'}{\longrightarrow} H^{-2,8}(\zk)\longrightarrow 0,\\
\label{eq_CH1_squareedge}
  &0\longrightarrow H^{-3,8}(\zk)\stackrel{d'}{\longrightarrow} H^{-2,6}(\zk)\stackrel{d'}{\longrightarrow} H^{-1,4}(\zk)\longrightarrow 0.
\end{align}
Since $d'$ of \eqref{eq_CH2_squareedge} is an isomorphism, we have 
$$
\HH^{-3,10}(\zk)=\HH^{-2,8}(\zk)=0.
$$
For the cochain complex \eqref{eq_CH1_squareedge}, observe that for $[u_1u_2u_3v_5]\in H^{-3,8}(\zk)$, 
\[
d'([u_1u_2u_3v_5])=[u_2u_3v_5]-[u_1u_3v_5]
+[u_1u_2v_5].
\]
The differential $d'$ from $H^{-2,6}(\zk)$ is described similarly, giving the following matrix presentations of the two differentials $\Z\to\Z^5$ and $\Z^5\to\Z^5$ in~\eqref{eq_CH1_squareedge}: 
\[
\begin{pmatrix}
1\\ 1\\ -1\\0 \\ 0 
\end{pmatrix} \text{ and } 
\begin{pmatrix}
0& 0& 0& -1& 0 \\
0& 0& 0& 0& -1\\
-1& 0& -1& 1& 0 \\
1& -1& 0& 0& 1\\
0& 1& 1& 0& 0 
\end{pmatrix}.
\]
Hence, the first differential is injective, while the second is surjective onto a direct summand of rank~$4$ in~$\Z^5$. The resulting nontrivial double cohomology groups are
\[
  \HH^{0,0}(\zk)\cong\HH^{-1,4}(\zk)\cong\Z,
\]
in accordance with Theorem~\ref{surgery}.
\end{example}

Here are two examples of simplicial complexes $\sK$ which are \emph{not} obtained by the procedure described in Theorem~\ref{surgery}. Nevertheless we have $\HH^*(\zk)\cong\Z\oplus\Z$ in bidegrees $(0,0)$ and $(-1,4)$. Both examples are wedge decomposable.

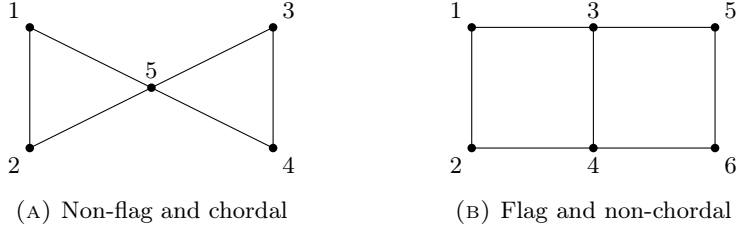
\begin{figure}
\begin{subfigure}{.45\linewidth}
\centering
  \begin{tikzpicture}[scale=0.8]
    
    \draw (0,0)--(0,2)--(2,1)--(4,0)--(4,2)--cycle;
    
    \fill [black] (0,2) circle (2pt) node[above left] at (0,2){\small 1};
    \fill [black] (0,0) circle (2pt) node[below left] at (0,0){\small 2};
    
    \fill [black] (2,1) circle (2pt) node[above] at (2,1){\small 5};
    
    \fill [black] (4,2) circle (2pt) node[above right] at (4,2){\small 3};
    \fill [black] (4,0) circle (2pt) node[below right] at (4,0){\small 4};
    
    \end{tikzpicture}
    \caption{Non-flag and chordal}
\end{subfigure}
\begin{subfigure}{.45\linewidth}
\centering
 \begin{tikzpicture}[scale=0.8]
    
    \draw (0,0)--(2,0)--(2,2)--(0,2)--cycle;
    \draw (2,2)--(4,2)--(4,0)--(2,0);
    
    \fill [black] (0,2) circle (2pt) node[above left] at (0,2){\small 1};
    \fill [black] (0,0) circle (2pt) node[below left] at (0,0){\small 2};
    
    \fill [black] (2,2) circle (2pt) node[above] at (2,2){\small 3};
    \fill [black] (2,0) circle (2pt) node[below] at (2,0){\small 4};
    
    \fill [black] (4,2) circle (2pt) node[above right] at (4,2){\small 5};
    \fill [black] (4,0) circle (2pt) node[below right] at (4,0){\small 6};
    
    \end{tikzpicture}
    \caption{Flag and non-chordal}
\end{subfigure}
\caption{Two examples}
\label{fig_twoexamples}
\end{figure}

\begin{example}\label{TwoTriangles}
Let $\sK$ be obtained by gluing two $3$-cycles 
$(1,2,5)$ and $(3,4,5)$ along the vertex $\{5\}$. See Figure \ref{fig_twoexamples}~(A). The nontrivial cohomology groups of $\zk$ are
\begin{align*}
  H^{0,0}(\zk)&\cong\Z\langle1\rangle,\\ 
  H^{-1,4}(\zk)&\cong\Z\langle[u_{1}v_{3}], [u_{1}v_{4}], [u_{2}v_{3}], [u_{2}v_{4}]\rangle,\\
  H^{-1,6}(\zk)&\cong\Z\langle[u_5v_3v_4],[u_5v_1v_2]\rangle,\\
  H^{-2,6}(\zk)&\cong\Z\langle[u_1u_2v_3], [u_1u_2v_4], [u_3u_4v_1], [u_3u_4v_2]\rangle,\\
  H^{-3,8}(\zk)&\cong\Z\langle[u_{1}u_{2}u_{3}v_{4}]-[u_{1}u_{2}u_{4}v_{3}]\rangle,\\
  H^{-2,8}(\zk)&\cong\Z\langle[u_1u_5v_3v_4],[u_2u_5v_3v_4],[u_3u_5v_1v_2],[u_4u_5v_1v_2]\rangle,\\
  H^{-3,10}(\zk)&\cong\Z\langle[u_1u_2u_5v_3v_4],[u_3u_4u_5v_1v_2]\rangle.
\end{align*} 
The nontrivial part of the complex $\CH^\ast(\zk)=(H^\ast(\zk), d')$ is given by
\begin{align}
\label{eq_CH1_wedge_of_triangles}
  0&\longrightarrow H^{-3,8}(\zk)\stackrel{d'}{\longrightarrow} H^{-2,6}(\zk)\stackrel{d'}{\longrightarrow} H^{-1,4}(\zk)\longrightarrow 0;\\
\label{eq_CH2_wedge_of_triangles}
  0&\longrightarrow H^{-3,10}(\zk)\stackrel{d'}{\longrightarrow} H^{-2,8}(\zk)\stackrel{d'}{\longrightarrow} H^{-1,6}(\zk)\longrightarrow 0.
\end{align}
Observe that $d' \colon H^{-3,8}(\zk) \to H^{-2,6}(\zk)$ in \eqref{eq_CH1_wedge_of_triangles} is given by 
\begin{align*}
    &d'([u_1u_2u_3v_4]-[u_1u_2u_4v_3])\\
    &\quad =[u_2u_3v_4]-[u_1u_3v_4]+[u_1u_2v_4]-[u_2u_4v_3]+[u_1u_4v_3]-[u_1u_2v_3]\\
    &\quad = -[u_3u_4v_2]+[u_3u_4v_1]+[u_1u_2v_4]-[u_1u_2v_3],
\end{align*}
where the second equality follows because 
\[
[u_3u_4v_2]-[u_2u_4v_3]+[u_2u_3v_4]=[u_1u_3v_4]-[u_1u_4v_3]+[u_3u_4v_1]=0 \in H^{-2,6}(\zk).
\]
Similarly, $d'\colon  H^{-2,6}(\zk
)\to  H^{-1,4}(\zk)$ is given by the matrix
\[
\begin{pmatrix}
-1 &0 &-1& 0 \\
0 &-1& 1& 0\\
1 &0 &0 &-1\\
0 &1 &0 &1
\end{pmatrix}
\]
with respect to the bases listed above. 
Its image is a direct summand of rank $3$ in $H^{-1,4}(\zk)$. It follows that 
\[
\HH^{-3,8}(\zk)=H^{-2,6}(\zk)=0 \quad\text{and}\quad \HH^{-1,4}(\zk)=\Z.
\]
Similar computations with the two differentials in \eqref{eq_CH2_wedge_of_triangles} 
show that
\[
\HH^{-3,10}(\zk)=\HH^{-2,8}(\zk)=\HH^{-1,6}(\zk)=0.
\]
Hence, the only nontrivial double cohomology groups are
\[
  \HH^{0,0}(\zk)\cong\HH^{-1,4}(\zk)\cong\Z.
\]
\end{example}

\begin{example}\label{TwoSquares}
Let $\sk$ be obtained by gluing two $4$-cycles $(1,2,3,4)$ and $(3,4,5,6)$ along the edge $(3,4)$. 
See Figure~\ref{fig_twoexamples}~(B). The nontrivial cohomology groups of $\zk$ are
\begin{align*}
  H^{0,0}(\zk)&\cong\Z\langle1\rangle,\\ 
  H^{-1,4}(\zk)&\cong\Z\langle[u_1v_4],[u_1v_5],[u_1v_6],[u_2v_3],[u_2v_5],[u_2v_6],[u_3v_6],[u_4v_5]\rangle,\\
  H^{-2,6}(\zk)&\cong\Z\langle[u_1u_2v_5],[u_1u_2v_6],[u_5u_6v_1],[u_5u_6v_2],[u_1u_3v_6],[u_2u_4v_5],\\
  & \qquad [u_3u_5v_2],[u_4u_6v_1],[u_1u_4v_5],[u_1u_5v_4],[u_2u_3v_6],[u_2u_6v_3]\rangle,\\
  H^{-3,8}(\zk)&\cong\Z\langle[u_1u_2u_3v_6],[u_1u_2u_4v_5],[u_3u_5u_6v_2],[u_4u_5u_6v_1],\\
  & \qquad [u_1u_2u_5v_6-u_1u_2u_6v_5]\rangle,\\
  H^{-2,8}(\zk)&\cong\Z\langle[u_3u_4v_1v_2],[u_3u_4v_5v_6]\rangle,\\
  H^{-3,10}(\zk)&\cong\Z\langle[u_3u_4u_5v_1v_2],[u_3u_4u_6v_1v_2],[u_1u_3u_4v_5v_6],[u_2u_3u_4v_5v_6]\rangle,\\
  H^{-4,12}(\zk)&\cong\Z\langle[u_3u_4u_5u_6v_1v_2],[u_1u_2u_3u_4v_5v_6]\rangle.
\end{align*} 
The nontrivial part of the complex $\CH^\ast(\zk)$ is given by
\begin{align*}
  &0\longrightarrow H^{-4,12}(\zk)\stackrel{d'}{\longrightarrow} H^{-3,10}(\zk)\stackrel{d'}{\longrightarrow} H^{-2,8}(\zk)\longrightarrow 0,\\
  &0\longrightarrow H^{-3,8}(\zk)\stackrel{d'}{\longrightarrow} H^{-2,6}(\zk)\stackrel{d'}{\longrightarrow} H^{-1,4}(\zk)\longrightarrow 0.
\end{align*}
A calculation similar to that in the previous example shows that the nontrivial double cohomology groups are
\[
  \HH^{0,0}(\zk)\cong\HH^{-1,4}(\zk)\cong\Z.
\]
\end{example}

We have considered several examples of $\sK$ 
with double cohomology of rank~$2$. Starting with these examples 
and applying Theorem~\ref{thm_join}, we can easily construct $\zk$ such that $\rank \HH^*(\zk)=2^n$ for each $n\geq 0$, and only the powers of two can be realized as ranks of the double cohomology in this way. 
Recall from Corollary \ref{cor_euler_char_of_HH_for_simplex} that the Euler characteristic of $\HH^\ast(\zk)$ is zero unless $\sK$ is a simplex. It follows that the rank of $\HH^\ast(\zk)$ is always even. Hence, we may ask the following question.

\begin{question}\label{HHrankRealization_quest}
Let $r$ be a positive even integer different from a power of two. Does there exist a simplicial complex $\sK$ such that $\rank \HH^*(\zk)=r$?
\end{question}

Recently Han \cite{han23} gave an example for $r=6$. By taking joins we can also get an example for any $r=2^{s+t}3^t$.

\end{document}